\documentclass[letterpaper,final,twoside,leqno]{amsart}

\usepackage[normalem]{ulem}
\usepackage{mathtools}
\usepackage{mathrsfs}  
\usepackage[arrow,curve,matrix]{xy}

\usepackage{scalerel}
\usepackage{comment}
\usepackage{xspace}

\newif\ifdraft
\draftfalse

\numberwithin{figure}{section}

\usepackage{amsmath,amsfonts,amsthm,mathrsfs}
\usepackage{amssymb}
\usepackage{mathrsfs}

\DeclareFontFamily{OMS}{rsfs}{\skewchar\font'60}
\DeclareFontShape{OMS}{rsfs}{m}{n}{<-5>rsfs5 <5-7>rsfs7 <7->rsfs10 }{}
\DeclareSymbolFont{rsfs}{OMS}{rsfs}{m}{n}
\DeclareSymbolFontAlphabet{\scr}{rsfs}

\usepackage[bookmarks]{hyperref}
\usepackage[usenames,dvipsnames]{xcolor}
\hypersetup{
    colorlinks=true,
    linkcolor={blue!50!black},
    citecolor={green!50!black},
    urlcolor={red!80!black}
}

\usepackage[alphabetic,initials]{amsrefs}

\usepackage{enumitem}

\usepackage{chngcntr}

\ifdraft
\usepackage[notcite,notref,color]{showkeys}

\definecolor{labelkey}{gray}{0.5}
\fi

\usepackage{tikz}

\usepackage{tikz-cd}
\tikzset{commutative diagrams/arrow style=math font}

\usetikzlibrary{matrix,arrows}
\newlength{\myarrowsize} 

\pgfarrowsdeclare{cmto}{cmto}{
	\pgfsetdash{}{0pt} 
	\pgfsetbeveljoin 
	\pgfsetroundcap 
	\setlength{\myarrowsize}{0.6pt}
	\addtolength{\myarrowsize}{.5\pgflinewidth}
	\pgfarrowsleftextend{-4\myarrowsize-.5\pgflinewidth} 
	\pgfarrowsrightextend{.8\pgflinewidth}
}{
	\setlength{\myarrowsize}{0.6pt} 
  	\addtolength{\myarrowsize}{.5\pgflinewidth}  
	\pgfsetlinewidth{0.5\pgflinewidth}
	\pgfsetroundjoin
	\pgfpathmoveto{\pgfpoint{1.5\pgflinewidth}{0}}
	\pgfpatharc{-109}{-170}{4\myarrowsize}
	\pgfpatharc{10}{189}{0.58\pgflinewidth and 0.2\pgflinewidth}
	\pgfpatharc{-170}{-115}{4\myarrowsize+\pgflinewidth}
	\pgfpathclose
	\pgfusepathqfillstroke
	\pgfpathmoveto{\pgfpoint{1.5\pgflinewidth}{0}}
	\pgfpatharc{109}{170}{4\myarrowsize}
	\pgfpatharc{-10}{-189}{0.58\pgflinewidth and 0.2\pgflinewidth}
	\pgfpatharc{170}{115}{4\myarrowsize+\pgflinewidth}
	\pgfpathclose
	\pgfusepathqfillstroke
	\pgfsetlinewidth{2\pgflinewidth}
}

\pgfarrowsdeclare{cmonto}{cmonto}{
	\pgfsetdash{}{0pt} 
	\pgfsetbeveljoin 
	\pgfsetroundcap 
	\setlength{\myarrowsize}{0.6pt}
	\addtolength{\myarrowsize}{.5\pgflinewidth}
	\pgfarrowsleftextend{-4\myarrowsize-.5\pgflinewidth} 
	\pgfarrowsrightextend{.8\pgflinewidth}
}{
	\setlength{\myarrowsize}{0.6pt} 
  	\addtolength{\myarrowsize}{.5\pgflinewidth}  
	\pgfsetlinewidth{0.5\pgflinewidth}
	\pgfsetroundjoin
	\pgfpathmoveto{\pgfpoint{1.5\pgflinewidth}{0}}
	\pgfpatharc{-109}{-170}{4\myarrowsize}
	\pgfpatharc{10}{189}{0.58\pgflinewidth and 0.2\pgflinewidth}
	\pgfpatharc{-170}{-115}{4\myarrowsize+\pgflinewidth}
	\pgfpathclose
	\pgfusepathqfillstroke
	\pgfpathmoveto{\pgfpoint{1.5\pgflinewidth}{0}}
	\pgfpatharc{109}{170}{4\myarrowsize}
	\pgfpatharc{-10}{-189}{0.58\pgflinewidth and 0.2\pgflinewidth}
	\pgfpatharc{170}{115}{4\myarrowsize+\pgflinewidth}
	\pgfpathclose
	\pgfusepathqfillstroke
	\pgfpathmoveto{\pgfpoint{1.5\pgflinewidth-0.3em}{0}}
	\pgfpatharc{-109}{-170}{4\myarrowsize}
	\pgfpatharc{10}{189}{0.58\pgflinewidth and 0.2\pgflinewidth}
	\pgfpatharc{-170}{-115}{4\myarrowsize+\pgflinewidth}
	\pgfpathclose
	\pgfusepathqfillstroke
	\pgfpathmoveto{\pgfpoint{1.5\pgflinewidth-0.3em}{0}}
	\pgfpatharc{109}{170}{4\myarrowsize}
	\pgfpatharc{-10}{-189}{0.58\pgflinewidth and 0.2\pgflinewidth}
	\pgfpatharc{170}{115}{4\myarrowsize+\pgflinewidth}
	\pgfpathclose
	\pgfusepathqfillstroke
	\pgfsetlinewidth{2\pgflinewidth}
}

\pgfarrowsdeclare{cmhook}{cmhook}{
	\pgfsetdash{}{0pt} 
	\pgfsetbeveljoin 
	\pgfsetroundcap 
	\setlength{\myarrowsize}{0.6pt}
	\addtolength{\myarrowsize}{.5\pgflinewidth}
	\pgfarrowsleftextend{-4\myarrowsize-.5\pgflinewidth} 
	\pgfarrowsrightextend{.8\pgflinewidth}
}{
	\setlength{\myarrowsize}{0.6pt} 
  	\addtolength{\myarrowsize}{.5\pgflinewidth}  
 	\pgfsetdash{}{0pt}
	\pgfsetroundcap
	\pgfpathmoveto{\pgfqpoint{0pt}{-4.667\pgflinewidth}}
	\pgfpathcurveto
    {\pgfqpoint{4\pgflinewidth}{-4.667\pgflinewidth}}
    {\pgfqpoint{4\pgflinewidth}{0pt}}
    {\pgfpointorigin}
	\pgfusepathqstroke
}

\newenvironment{diagram*}[2]{%
\[%
\begin{tikzpicture}[>=cmto,baseline=(current bounding box.center),%
	to/.style={->,font=\scriptsize,cap=round},%
	into/.style={cmhook->,font=\scriptsize,cap=round},%
	onto/.style={-cmonto,font=\scriptsize,cap=round},%
	math/.style={matrix of math nodes, row sep=#2, column sep=#1,%
		text height=1.5ex, text depth=0.25ex}]%
}{%
\end{tikzpicture}%
\]%
\ignorespacesafterend%
}

\DeclareMathOperator{\disc}{disc}
\DeclareMathOperator{\pr}{pr}

\newcommand{\theoremref}[1]{\hyperref[#1]{Theorem~\ref*{#1}}}
\newcommand{\lemmaref}[1]{\hyperref[#1]{Lemma~\ref*{#1}}}
\newcommand{\definitionref}[1]{\hyperref[#1]{Definition~\ref*{#1}}}
\newcommand{\propositionref}[1]{\hyperref[#1]{Proposition~\ref*{#1}}}
\newcommand{\conjectureref}[1]{\hyperref[#1]{Conjecture~\ref*{#1}}}
\newcommand{\corollaryref}[1]{\hyperref[#1]{Corollary~\ref*{#1}}}
\newcommand{\exampleref}[1]{\hyperref[#1]{Example~\ref*{#1}}}

\makeatletter
\let\old@caption\caption
\renewcommand*{\caption}[1]{%
  \setcounter{figure}{\value{equation}}%
  \stepcounter{equation}%
  \old@caption{#1}\relax%
}
\makeatother

\newcounter{intro}

\newtheorem{intro-conjecture}[intro]{Conjecture}
\newtheorem{intro-corollary}[intro]{Corollary}
\newtheorem{intro-theorem}[intro]{Theorem}

\newcommand{\parref}[1]{\hyperref[#1]{\S\ref*{#1}}}

\makeatletter
\newcommand*\if@single[3]{%
  \setbox0\hbox{${\mathaccent"0362{#1}}^H$}%
  \setbox2\hbox{${\mathaccent"0362{\kern0pt#1}}^H$}%
  \ifdim\ht0=\ht2 #3\else #2\fi
  }
\newcommand*\rel@kern[1]{\kern#1\dimexpr\macc@kerna}
\newcommand*\widebar[1]{\@ifnextchar^{{\wide@bar{#1}{0}}}{\wide@bar{#1}{1}}}
\newcommand*\wide@bar[2]{\if@single{#1}{\wide@bar@{#1}{#2}{1}}{\wide@bar@{#1}{#2}{2}}}
\newcommand*\wide@bar@[3]{%
  \begingroup
  \def\mathaccent##1##2{%
    \if#32 \let\macc@nucleus\first@char \fi
    \setbox\z@\hbox{$\macc@style{\macc@nucleus}_{}$}%
    \setbox\tw@\hbox{$\macc@style{\macc@nucleus}{}_{}$}%
    \dimen@\wd\tw@
    \advance\dimen@-\wd\z@
    \divide\dimen@ 3
    \@tempdima\wd\tw@
    \advance\@tempdima-\scriptspace
    \divide\@tempdima 10
    \advance\dimen@-\@tempdima
    \ifdim\dimen@>\z@ \dimen@0pt\fi
    \rel@kern{0.6}\kern-\dimen@
    \if#31
      \overline{\rel@kern{-0.6}\kern\dimen@\macc@nucleus\rel@kern{0.4}\kern\dimen@}%
      \advance\dimen@0.4\dimexpr\macc@kerna
      \let\final@kern#2%
      \ifdim\dimen@<\z@ \let\final@kern1\fi
      \if\final@kern1 \kern-\dimen@\fi
    \else
      \overline{\rel@kern{-0.6}\kern\dimen@#1}%
    \fi
  }%
  \macc@depth\@ne
  \let\math@bgroup\@empty \let\math@egroup\macc@set@skewchar
  \mathsurround\z@ \frozen@everymath{\mathgroup\macc@group\relax}%
  \macc@set@skewchar\relax
  \let\mathaccentV\macc@nested@a
  \if#31
    \macc@nested@a\relax111{#1}%
  \else
    \def\gobble@till@marker##1\endmarker{}%
    \futurelet\first@char\gobble@till@marker#1\endmarker
    \ifcat\noexpand\first@char A\else
      \def\first@char{}%
    \fi
    \macc@nested@a\relax111{\first@char}%
  \fi
  \endgroup
}
\makeatother

\usepackage{amsbsy}

\usepackage{marvosym}

\DeclareMathAlphabet{\smallchanc}{OT1}{pzc}%
                                 {m}{it}
\DeclareFontFamily{OT1}{pzc}{}
\DeclareFontShape{OT1}{pzc}{m}{it}%
             {<-> s * [1.100] pzcmi7t}{}
\DeclareMathAlphabet{\mathchanc}{OT1}{pzc}%
                                 {m}{it}

\newcommand{\mcM}{\mathchanc{M}}

\newcommand{\sA}{\mathscr{A}}
\newcommand{\sB}{\mathscr{B}}

\newcommand{\sE}{\mathscr{E}}
\newcommand{\sF}{\mathscr{F}}
\newcommand{\sG}{\mathscr{G}}
\newcommand{\sH}{\mathscr{H}}

\newcommand{\sL}{\mathscr{L}}
\newcommand{\sM}{\mathscr{M}}
\newcommand{\sN}{\mathscr{N}}
\newcommand{\sO}{\mathscr{O}}

\newcommand{\bC}{\mathbb{C}}

\newcommand{\bN}{\mathbb{N}}

\newcommand{\bP}{\mathbb{P}}
\newcommand{\bQ}{\mathbb{Q}}

\newcommand{\bZ}{\mathbb{Z}}

\DeclareSymbolFont{largesymbolsA}{U}{jkpexa}{m}{n}
\SetSymbolFont{largesymbolsA}{bold}{U}{jkpexa}{bx}{n}
\DeclareMathSymbol{\varprod}{\mathop}{largesymbolsA}{16}
\makeatletter
\newcommand{\LeftEqNo}{\let\veqno\@@leqno}
\makeatother
\newcommand\exist{\exists\,}

\newcommand{\ol}{\overline}

\newcommand{\into}{\hookrightarrow}

\newcommand{\properideal}%
        {\subsetneq}

\newcommand{\leteq}{\colon\!\!\!=}

\newcommand{\ratmap}{\dasharrow}

\newcommand\dash[1]{\rule[-.2ex]{#1}{.4pt}}

\DeclareMathOperator{\Aut}{Aut}

\DeclareMathOperator{\codim}{codim}

\DeclareMathOperator{\Hilb}{Hilb}

\DeclareMathOperator{\Ob}{{Ob}}

\DeclareMathOperator{\red}{red}

\DeclareMathOperator{\supp}{{supp}}
\DeclareMathOperator{\sym}{{Sym}}

\newcommand{\factor}[2]{\left. \raise .2em\hbox{\ensuremath{#1}\vphantom{$I^d$}}
\hskip -.1em \right/ \hskip -.4em \raise -.3em\hbox{\ensuremath{#2}}}%
\newcommand\mtimes[3]{{\varprod_{#1}^{#2}}_{\raise 1ex \hbox{\scriptsize #3}}}%

\newcommand{\kdot}{{{\,\begin{picture}(1,1)(-1,-2)\circle*{2}\end{picture}\,}}}

\newcommand{\dcx}[1]{{\omega}^\kdot_{#1}}

\def\dimcoh#1.#2.#3.{h^{#1}(#2,#3)}
\def\hypcoh#1.#2.#3.{\mathbb H_{\vphantom{l}}^{#1}(#2,#3)}
\def\loccoh#1.#2.#3.#4.{H^{#1}_{#2}(#3,#4)}
\def\dimloccoh#1.#2.#3.#4.{h^{#1}_{#2}(#3,#4)}
\def\lochypcoh#1.#2.#3.#4.{\mathbb H^{#1}_{#2}(#3,#4)}
\def\seslong#1.#2.#3.{0  \longrightarrow  #1   \longrightarrow 
 #2 \longrightarrow #3 \longrightarrow 0} 
\def\sesshort#1.#2.#3.{0
 \rightarrow #1 \rightarrow #2 \rightarrow #3 \rightarrow 0}
\def\Iff#1#2#3{
\hfil\hbox{\hsize =#1
\vtop{\noin #2}
\hskip.5cm 
\lower.5\baselineskip\hbox{$\Leftrightarrow$}\hskip.5cm
\vtop{\noin #3}}\hfil\medskip}
\newcommand{\union}\cup
\newcommand{\intersect}\cap
\newcommand{\Union}\bigcup
\newcommand{\Intersect}\bigcap
\def\myoplus#1.#2.{\underset #1 \to {\overset #2 \to \oplus}}

\newcommand{\resto}[1]{{\raise -.5ex\hbox{$\vert$}_{#1}}}

\newcommand\noin{\noindent}

   [2017/06/13 v0.1 finetuning mabliautoref and theorem setup]

\RequirePackage{hyperref}

\newcommand{\sectionsize}{} 
\newcommand{\theoremsize}{} 

\renewcommand{\subsectionautorefname}{\sectionsize\sf \subsectionautorefname}

\makeatletter
\@ifdefinable\equationname{\let\equationname\equationautorefname}
\def\equationautorefname~#1\@empty\@empty\null{\protect{\theoremsize
    (#1\@empty\@empty\null)}}%
\@ifdefinable\AMSname{\let\AMSname\AMSautorefname}
\def\AMSautorefname~#1\@empty\@empty\null{
  ( #1\@empty\@empty\null)}%
\@ifdefinable\itemname{\let\itemname\itemautorefname}
\def\itemautorefname~#1\@empty\@empty\null{\theoremsize{%
    {#1}}\@empty\@empty\null%
}%
\makeatother

\RequirePackage{amsthm}
\RequirePackage{aliascnt}
\newcommand{\basetheorem}[3]{%
    \newtheorem{#1}{#2}[#3]
    \newtheorem*{#1*}{#2}
    \expandafter\def\csname #1autorefname\endcsname{#2}
}%
\newcommand{\maketheorem}[3]{%
    \newaliascnt{#1}{#2}
    \newtheorem{#1}[#1]{\theoremsize #3}
    \aliascntresetthe{#1}
    \expandafter\def\csname #1autorefname\endcsname{#3}
    \newtheorem{#1*}{#3}
}%
\newcommand{\baseremark}[3]{%
    \newtheorem{#1}{#2}{#3}
    \newtheorem*{#1*}{#2}
    \expandafter\def\csname #1autorefname\endcsname{#2}
}%
\newcommand{\makeremark}[3]{%
    \newaliascnt{#1}{#2}
    \newtheorem{#1}[equation]{#3}
    \aliascntresetthe{#1}
    \expandafter\def\csname #1autorefname\endcsname{\theoremsize\sf #3}
    \newtheorem{#1*}{#3}
}%
\theoremstyle{plain}   

\basetheorem{theorem}{Theorem}{section}

\theoremstyle{definition}    

\newcommand{\hide}[1]{}

\DeclareFontFamily{OMS}{rsfs}{\skewchar\font'60}
\DeclareFontShape{OMS}{rsfs}{m}{n}{<-5>rsfs5 <5-7>rsfs7 <7->rsfs10 }{}
\DeclareSymbolFont{rsfs}{OMS}{rsfs}{m}{n}
\DeclareSymbolFontAlphabet{\scr}{rsfs}

\setlist[enumerate, 1]{font=\upshape}

\numberwithin{figure}{section}

\DeclareMathOperator{\Gal}{Gal}
\DeclareMathOperator{\KSBm}{KSB}
\DeclareMathOperator{\Sl}{Sl}

\DeclareMathOperator{\Spec}{Spec}
\DeclareMathOperator{\vol}{vol}

\newcommand{\wtilde}{\widetilde}

\newcommand{\sblank}{\dash{.6em}}
\newcommand{\ksb}{KSB-stable\xspace}
\newcommand{\lst}{locally KSB-stable\xspace}

\theoremstyle{plain}
\maketheorem{prop}{theorem}{Proposition}
\maketheorem{cor}{theorem}{Corollary}
\maketheorem{lem}{theorem}{Lemma}
\maketheorem{conj}{theorem}{Conjecture}

\newtheorem{proposition}[prop]{Proposition}
\newtheorem{corollary}[cor]{Corollary}
\newtheorem{lemma}[lem]{Lemma}

\theoremstyle{definition}
\maketheorem{defini}{theorem}{Definition}
\newtheorem{definition}[defini]{Definition}
\theoremstyle{remark}
\maketheorem{rem}{theorem}{Remark}
\newtheorem{remark}[rem]{Remark}
\theoremstyle{example}
\maketheorem{exa}{theorem}{Example}
\newtheorem{example}[exa]{Example}

\maketheorem{notat}{theorem}{Notation}
\maketheorem{notr}{theorem}{Notation-Remark}
\maketheorem{defnot}{theorem}{Definition-Notation}
\maketheorem{prep}{theorem}{Preparation}
\newtheorem{notation}[notat]{Notation}
\newtheorem{not-rem}[notr]{Notation-Remark}
\newtheorem{def-not}[defnot]{Definition-Notation}

\maketheorem{set}{theorem}{Set-up}
\newtheorem{set-up}[set]{Set-up}

\maketheorem{term}{theorem}{Terminology}
\newtheorem{terminology}[term]{Terminology}

\maketheorem{back}{theorem}{Background}

\maketheorem{observ}{theorem}{Observation}

\maketheorem{fac}{theorem}{Fact}

\maketheorem{ass}{theorem}{Assumption}
\newtheorem{assumption}[ass]{Assumption}

\maketheorem{sumnot}{theorem}{Summary-Notation}
\newtheorem{summary-notation}[sumnot]{Summary-Notation}

\maketheorem{deftheo}{theorem}{Definition-Theorem}
\newtheorem{definition-theorem}[deftheo]{Definition-Theorem}

\maketheorem{conn}{theorem}{Convention}

\maketheorem{clai}{theorem}{Claim}
\newtheorem{claim}[clai]{Claim}

\maketheorem{sclai}{theorem}{Subclaim}

\setlist[enumerate]{label=(\thetheorem.\arabic*), before={\setcounter{enumi}{\value{equation}}}, after={\setcounter{equation}{\value{enumi}}}}

\numberwithin{equation}{theorem}

\makeatletter
\let\amsmath@bigm\bigm

\renewcommand{\bigm}[1]{%
  \ifcsname fenced@\string#1\endcsname
    \expandafter\@firstoftwo
  \else
    \expandafter\@secondoftwo
  \fi
  {\expandafter\amsmath@bigm\csname fenced@\string#1\endcsname}%
  {\amsmath@bigm#1}%
}
\newcommand{\DeclareFence}[2]{\@namedef{fenced@\string#1}{#2}}
\makeatother

\DeclareFence{\mid}{|}

\newcommand\Yi{Y_i}
\newcommand\Yii[1]{Y_{#1}}
\newcommand\Wi{W_i}
\newcommand\Wii[1]{W_{#1}}
\newcommand\pii{\pi_i}
\newcommand\taui{\tau_i}
\newcommand\etai{\eta_i}
\newcommand\Lambdai[1]{\Lambda_{#1}}
\newcommand\Xii{{X_i}}
\newcommand\sLi{{\sL_i}}
\newcommand\fii{{f_i}}

\begin{document}

\title[Asymptotic GIT-stability of KSB-stable varieties]%
{Asymptotic GIT-stability of KSB-stable varieties}

\author{S\'andor J Kov\'acs} \address{S\'andor Kov\'acs, Dept.~of Mathematics,
  University of Washington, Seattle, WA, 98195, USA}
\email{\href{mailto:skovacs@uw.edu }{skovacs@uw.edu }}
\urladdr{\href{http://sites.math.washington.edu/~kovacs/current/}
  {http://sites.math.washington.edu/~kovacs/current/}}

\author{Behrouz Taji} \address{Behrouz Taji, School of Mathematics and Statistics,
The University of New South Wales Sydney, NSW 2052 Australia}
\email{\href{mailto:b.taji@unsw.edu.au}{b.taji@unsw.edu.au}}
\urladdr{\href{https://web.maths.unsw.edu.au/~btaji//}
  {https://web.maths.unsw.edu.au/~btaji/}}

\thanks{Both authors were supported by Australian Research Council Discovery 
Project grant DP240101934 ``Deformation 
of Singularities through Hodge Theory and Derived Categories".\\
SJK was supported in part by NSF Grant DMS-2100389, and the Craig McKibben
and Sarah Merner Endowed Professorship in Mathematics}




\setlength{\parskip}{0.19\baselineskip}


\begin{abstract}
  We prove that KSB-stable varieties introduced by Koll\'ar and Shepherd-Barron as a
  generalization of Deligne--Mumford's stable curves are stable in the sense of
  geometric invariant theory (GIT).
  More precisely, we show that KSB-stable varieties are asymptotically GIT-stable with 
  respect to a sequence of linearizations defined by the Hodge bundle. 
\end{abstract}

\maketitle


\newcommand\hmarginpar[1]{}

\section{Introduction}
\label{s:intro}

Construction and study of moduli spaces of varieties is at the heart of
classification theory in algebraic geometry.  As it parametrizes isomorphism classes
of objects sharing prescribed properties, a moduli space is usually constructed as
the orbit space of a group action on a certain quasi-projectve scheme $H$ naturally
embedded in a Hilbert scheme.

To construct such a quotient space in the category of quasi-projective schemes
Mumford developed \emph{Geometric Invariant Theory} (GIT).  In particular he showed
that if all closed points of such a scheme $H$ are 
\emph{GIT-semistable}, then the desired moduli space exists as a \emph{GIT-quotient},
which is a projective scheme.  The aim of this paper is to construct such quotient
spaces for an important class of varieties in characteristic zero.

While it is relatively straightforward to define what GIT-stability means, it is much
less straightforward 
to test whether a given point is GIT-stable. Here stability is measured with respect
to an ample line bundle called a \emph{linearization}. By choosing the \emph{standard
  linearization}--or rather a family of standard linearizations--that naturally
exists on a sequence of parametrizing schemes $H$ (as subschemes of Hilbert schemes)
this problem was solved in several major cases: The case of \emph{stable curves},
where the resulting coarse moduli space is denoted by $\overline M_g$
\cite{MR1304906}, \cite{Mum77}, \cite{Gie82}, and for abelian varieties
\cite{MR1304906}.

Generalizing this approach to higher dimensions ran into several obstacles. One of
the fundamental problems was to find the correct higher dimensional analogue of
Deligne-Mumford's stable curves.  This was achieved in a seminal paper by Koll\'ar
and Shepherd-Barron \cite{Ksb88}.  These higher dimensional analogues are now
referred to as KSB-stable varieties.

In higher dimensions, Mumford's choice of linearization, the so-called standard
linearization (see \autoref{eq:standard} for the precise description), was used to
prove GIT-stability for smooth stable varieties, i.e., canonically polarized
projective manifolds.  
The existence of (unique) smooth K\"ahler--Einstein metrics \cite{MR0451180} on those
manifolds leads to their \emph{asymptotic (GIT-)stablity} (\autoref{def:asym})
\cite{Don01}, \cite{Luo98}, \cite{Wan04}. 
However, in the singular case, again with respect to Mumford's standard linearization, it was shown
by Wang and Xu \cite{WX14}, using \cite[Prop.~3.4]{Mum77}, that KSB-stable varieties
are \emph{not} asymptotically semistable.  Finding suitable choices of linearizations
for which all KSB-stable varieties (with fixed dimension and volume) are
GIT-semistable has so far proven to be elusive \cite[\S 1.2.1]{Laz14}.

In the absence of a GIT approach one may follow an alternative 
construction for the coarse moduli space $\mcM^{\KSBm}_{d, \nu}$ of KSB-stable
varieties of dimension $d$ and volume $\nu$ as a projective scheme
\cite{ModBook}.  
The first step is to show that the corresponding Deligne--Mumford stack
$\mathcal M^{\KSBm}_{d, \nu}$ is of finite type and proper. This relies, among other
things, on multiple deep results in birational geometry, including
\cite{BCHM10}. 
These properties of $\mathcal M^{\KSBm}_{d,\nu}$ then allow us to employ the results
of Keel--Mori \cite{KM97} and Koll\'ar \cite{Kollar97a} to realize
$\mcM^{\KSBm}_{d, \nu}$ as a proper algebraic space of finite type.  Once this is
achieved, the next step is to use Koll\'ar's projectivity criterion for complete
moduli spaces \cite{Kollar90,Fuj18,KP17} (see \autoref{thm:Kol}) together with
positivity properties of direct image sheaves $f_*\omega^{[k]}_{X/B}$ (see
\autoref{ff}) due to Kawamata \cite{Kaw81}, Viehweg \cite{Vie83,MR996558}, Koll\'ar
\cite{Kollar87}, and Fujino \cite{Fuj18} to prove that $\mcM^{\KSBm}_{d, \nu}$ is a
projective scheme.  The latter is accomplished by finding that
\begin{equation}\label{eq:det}
\det\big(f_*\omega^{[k]}_{X/B}  \big)
\end{equation} 
descends to a $\bQ$-ample line bundle on $\mcM^{\KSBm}_{d, \nu}$ (for sufficiently
large and divisible $k$).

On the other hand, 
in light of Fujino's semipositivity result \cite{Fuj18}, one would still expect that
$\mcM^{\KSBm}_{d, \nu}$ can be realized as a GIT-quotient via linearizations in terms
of \autoref{eq:det} for the embedded universal KSB-stable family over the associated
Hilbert scheme (see \autoref{eq:KSBlin} and \S.\ref{ss:Hilbert}).  Our main result is
finding an explicit sequence of linearizations (see \autoref{eq:KSBlin} for the
precise description) that verifies this prediction.






\begin{theorem}\label{thm:main1}
  There exists a sequence of linearizations with respect to which all KSB-stable
  varieties are asymptotically (GIT)-stable. In particular $\mcM^{\KSBm}_{d, \nu}$
  can be realized as a GIT-quotient.
\end{theorem}

As mentioned already, this is not based on Mumford's standard linearization \autoref{eq:standard}, which in
turn uses the classical Pl\"ucker embeddings. In particular, \autoref{thm:main1} is
not in contradiction with the results of \cite{WX14}.

In this paper a moduli functor is a contravariant functor from the category of
schemes to sets, assigning (isomorphism classes of) polarized, flat and projective
families satisfying certain properties of schemes (see \S \ref{ss:Good} for the
details).  We refer to \autoref{KSB} for the definition of the moduli functor
$\sM^{\KSBm}_{d, \nu}$ of KSB-stable varieties.

\subsection{An alternative GIT approach}
Our approach to find a sequence of linearizations that differ from Mumford's is based
on combining a GIT-stability criterion of Viehweg \cite[Thm.~4.25]{Viehweg95} and a
criterion of Koll\'ar for the projectivity of algebraic coarse moduli spaces
\cite{Kollar90} of functorially polarized moduli functors of $\bQ$-polarized schemes,
see \S\ref{ss:Good} and \autoref{def:functoial}.  The latter uses the following
\emph{semipositivity} notion for functors.

\begin{definition}[\protect{\cite[Def.~2.4]{Kollar90}}]\label{def:SP}
  A moduli functor $\sM$ of ($\bQ$-)polarized schemes is \emph{semipositive}, if
  there is a $d_0\in \bN$ such that for every smooth complete curve $C$ and any
  $\big( f: (X, \sL) \to C \big)\in \sM(C)$, $f_*\sL^{[d]}$ is nef, for every
  $d \geq d_0$.
\end{definition}

\begin{theorem}[\protect{Koll\'ar's projectivity criterion \cite[Thm.~2.6]{Kollar90}}]
\label{thm:Kol}
A good, functorially polarized, and semipositive moduli functor admits a projective coarse
moduli scheme.
\end{theorem}

Following \cite{ModBook}, a \emph{good moduli functor} is one that is separated,
bounded (\autoref{def:bounded}), proper \cite[Def.~ 1.15]{Viehweg95}, consists only
of objects with finite automorphism groups and is parametrized by a constructible
subscheme of a corresponding Hilbert scheme (see \autoref{def:openclosed} and \S
\ref{ss:Hilbert}). Note that this notion differs from Alper's good moduli spaces
\cite{MR3237451}.

Thanks to major breakthroughs in birational geometry, e.g.,
\cite{BCHM10}, 
\cite{HMX18}, 
\cite{HX13} a prime example of a good moduli functor is $\sM^{\KSBm}_{d,
  \nu}$. Furthermore by fundamental results of \cite{Kaw81}, \cite{Cam04} and
\cite[Thm.~1.2]{Fuj18} it is known that $\sM^{\KSBm}_{d, \nu}$ is semipositive.
Therefore, by applying \autoref{thm:Kol}, it follows that $\mcM^{\KSBm}_{d, \nu}$
exists as a projective scheme cf.~\cite{Fuj18}, \cite{KP17}.

In this paper we establish \autoref{thm:main1} as an instance of the following
GIT-analogue of \autoref{thm:Kol}.

\begin{theorem}\label{thm:main2}
  The objects of 
  any good and functorially polarized semipositive moduli functor $\sM$
  are asymptotically stable (with respect to \autoref{eq:lin}).
  In particular $\sM$ is coarsely
  representable by a GIT-quotient.
\end{theorem}

\autoref{thm:main2} is a consequence of the following GIT-stability criterion based
on Viehweg's use \cite[Thm.~4.25]{Viehweg95} of the notion of \emph{weak positivity}
(\autoref{def:WP}).

\subsubsection{A stability criterion for parameter schemes of bounded and separated
  moduli functors; \autoref{thm:ViehwegStability}.} \label{SCIntro} Let $H$ be a
quasi-projective subscheme of the Hilbert scheme $\Hilb^{r-1}$ parametrizing objects
of a functorially polarized, separated and bounded moduli functor $\sM$ embedded in
$\mathbb P^{r-1}$ (see \S \ref{ss:Hilbert}).  Then the quotient of $H$ by the the
special linear group arising from change of coordinates in $\bP^{r-1}$ forms a
quasi-projective coarse moduli scheme (a geometric quotient), if the following
condition is satisfied.

\noindent\refstepcounter{equation}\theequation\label{Crit}.
\emph{There exists a $k_0>0$ such that for any reduced quasi-projective scheme
  $Y$, 
    \begin{equation*}\label{eq:SP}
      \big(  f: (X, \sL) \to Y \big)  \in \sM(Y)  \implies
      \begin{minipage}[r]{.5\linewidth}
        $f_*\sL^{[k]}$ is weakly positive over $Y$, \\ for every $k$ divisible by
        $k_0$.
      \end{minipage}
    \end{equation*}}
\smallskip
\noindent We refer to \autoref{thm:ViehwegStability} for the specific choice of
linearizations with respect to which this GIT-stability criterion holds.

With this criterion, \autoref{thm:main2} is a direct consequence of the following
result.

\begin{theorem}\label{thm:WPkey}
  Let $\sM$ be a good moduli functor.  If $\sM$ is semipositive, then there is a
  $k_0>0$ satisfying {\rm\autoref{Crit}}
\end{theorem}

In particular as a consequence of \autoref{thm:WPkey} and the semipositivity of
$\sM^{\KSBm}_{d,\nu}$ \cite[Thm.~1.2]{Fuj18}, for KSB-stable families, the following
result extends \cite[Thm.~6.22]{Viehweg95}, where fibers were assumed to be smooth,
and \cite[Thm.~1.7]{Fuj18} 
to arbitrary reduced quasi-projective base schemes.  This is what is needed to apply
the GIT-stability criterion \autoref{Crit} for $\sM^{\KSBm}_{d, \nu}$, and therefore
to prove \autoref{thm:main1} using \autoref{thm:main2}.
\begin{corollary}\label{cor:WP}
  Given $d, \nu >0$, there is an integer $k_0>0$ such that for every KSB-stable family
  $f:X\to B$, over any reduced quasi-projective scheme $B$, with $\vol(X_b)= \nu$, and $\dim(X_b)=d$, $f_*\omega_{X/B}^{[k]}$ is
  weakly positive over $B$, for every $k$ divisible by $k_0$.
\end{corollary}

\subsection{Overview of the strategy of the proof of the main results}
We first explain the main ingredients of the proof of \autoref{thm:main2} in the
special setting of \autoref{cor:WP}.  When $B$ is normal this follows directly from
\cite{Kaw81}, \cite{Viehweg83} and \cite{Fuj18} for a desingularization
$\mu: \widehat B\to B$ of $B$ and the pullback $\widehat f: \widehat X\to \widehat B$
of $f$ via $\mu$. More precisely, for sufficiently large and divisible $k$, one uses
the isomorphism
\begin{equation}\label{eq:STAR}
\mu^* f_*\omega_{X/B}^{[k]} \simeq \widehat f_*\omega^{[k]}_{\widehat X/\widehat B},
\end{equation}
(see \S.\ref{FAMS} for notation), together with a descent property of weakly positive
sheaves to normal schemes (see \autoref{prop:Trace}), i.e., weak positivity of the
right-hand side of \autoref{eq:STAR} (\cite{Fuj18}) implies that of
$f_*\omega^{[k]}_{X/B}$. However when $B$ is not normal this property fails in
general, \emph{unless $B$ is proper} and the sheaves are locally free (\autoref{rem:nef}).  This
constitutes a major obstacle in proving \autoref{cor:WP}.  It turns out that
properness of $\sM^{\KSBm}_{d, \nu}$ plays a key role in overcoming this difficulty.
Although substantially easier, the proof in the one-dimensional base case already
elucidates how properness of the functor (or the stack) can be used for establishing
\autoref{cor:WP}.

\subsubsection{One-dimensional base case}\label{ss:curve}
Assuming that $\dim(B)=1$, let $\overline f: \overline X\to \overline B$ be a
projective compactification of $f: X\to B$ and $\pi_1: W_1 \to \overline B$
the normalization of $\overline B$, with $f_1: X_1\to W_1$ being the
pullback family. Let $f': X' \to W'$ be its stable reduction through the finite
morphism $\tau: W'\to W_1$, i.e., $f'$ is $\KSBm$-stable (\S.\ref{KSBFAMS}) and
only differs (up to isomorphism over $W'$) from the pullback of $\overline f$ over
$D:= \overline B\setminus B$ (\autoref{def:reduction}).

\subsubsection*{A flat base change and an alternative reduction}
We next consider finding another stable reduction after a certain flat base change.
Namely we construct a finite covering $\sigma: V\to \overline B$ that is flat and
generically coincides with a rational map $\eta: V\dashrightarrow W'$ composed with
$\pi_1\circ \tau$ (see \autoref{lem:covering} for the details of this construction):
\begin{equation}\label{curve}
  \begin{aligned}
    \xymatrix{
      &&&&  W' \ar[d]^{\tau}  && \ar[ll]_{f'}  X' \\
      \wtilde X \ar[d]_{\wtilde f}  &&&&  W_1\ar[d]^{\pi_1} \\
      \wtilde V \ar[r]^{\pi''} \ar@/^5mm/[rrrruu]^{\wtilde \eta} & V'\ar[r]^{\pi'} &
      V \ar[rr]^{\sigma} \ar@{-->}[rruu]^{\eta} \ar[rr]^{\tau} && \overline B &&
      \overline X \ar[ll]_{\overline f} } .
  \end{aligned}
\end{equation}
Now let $\pi: \wtilde V\to V$ be the normalization of $V$, naturally factoring
through the normalization $\pi': V'\to V$ of $V$ strictly over $\sigma^{-1}(D)$
via $\pi'': \wtilde V\to V'$.

\smallskip
\noindent
\emph{A functorial extension of $f_*\omega_{X/B}^{[k]}$ after flat base change:}
Using the fact that $f_*\omega_{X/B}^{[k]}$ and $f'_*\omega_{X'/W'}^{[k]}$ commute
with base change (\autoref{rem:ABC}), we find that
\[
  (\pi')^*\sigma^* f_*\omega^{[k]}_{X/B} \;\;\; \text{and} \;\;\; (\eta')^*
  f_*\omega_{X'/W'}^{[k]}
\]
glue to form a locally free extension $\sF'$ of the former sheaf on $V'$ such that
\begin{equation}\label{eq:compatible}
(\pi'')^* \sF' \simeq \wtilde f_*\omega^{[k]}_{\wtilde X/\wtilde V} . 
\end{equation}
Now, by \cite{Kaw81}, \cite{Fuj18}, and the fact that $\wtilde V$ is proper, we know
that the right-hand side of \autoref{eq:compatible} is nef (see
\autoref{rem:nef}). Therefore so is $\sF'$.  In particular, after restricting to
$\sigma^{-1}(B)$ we find that $\sigma^*f_*\omega^{[k]}_{X/B}$ is weakly positive.
With $\sigma$ being flat, weak positivity descends, leading to the weak positivity of
$f_*\omega^{[k]}_{X/B}$ (see \autoref{item:splitflat}).

\subsubsection{Higher dimensional base}\label{higher}
Here the problem is considerably more subtle. In particular 
it is not clear how the extension of $f_*\omega^{[k]}_{X/B}$, similar to $\sF'$ in \ref{ss:curve} \autoref{eq:compatible}, 
over a proper compactification of a covering can be constructed. 
We achieve this by the following steps: 
\begin{enumerate}
\item\label{item:extend} We use the existence of the coarse space
  $\mcM^{\KSBm}_{d, \nu}$ of $\mathcal M_{d, \nu}^{\KSBm}$ as a proper algebraic space \cite{Kollar97a}, \cite{KM97},
  combined with an extension result for Deligne--Mumford
  stacks~\cite[Lem.~2.2.3]{AV02}, to construct a \emph{stratified stable reduction}
  together with a proper scheme $V$ equipped with a flat covering
  $\sigma: V\to \overline B$ as in \autoref{curve}.  This is the content of
  \autoref{lem:stratified}.

\item We then use \autoref{item:extend} and an extension theorem due to Gabber (see
  \S.\ref{GAB}) to effectively reduce the construction of the desired extension of
  $f_*\omega^{[k]}_{X/B}$ to its extension over discrete valuation rings. We note
  that this extension may require replacing $V$ by another birational model (similar
  to $V'$ in \autoref{curve}).  This is done in \S \ref{ss:WPkey}.
\end{enumerate}

To prove \autoref{thm:main1} we then apply \autoref{cor:WP} to the criterion
\S\ref{SCIntro}, or more specifically \autoref{thm:ViehwegStability}.  As a consequence, after fixing a finite
number of components of $\Hilb^{r-1}$ determined by $d$ and $\nu$, we find that for
every sufficiently large and divisible $k>0$, with the corresponding (universally embedded)
KSB-stable family $f: \mathfrak X\to H^{r,k}\subset \Hilb^{r-1}$, there is
a $\mu_0:= \mu_0(k)>0$ such that 
the GIT-stable locus of
\begin{equation}\label{eq:KSBlin}
  \det(f_*\omega^{[k]}_{\mathfrak X/H^{r,k}}) \otimes (\det f_*\omega_{\mathfrak
    X/H^{r,k}}^{[k\mu]})^r ,  
\end{equation}
is the entire $H^{r,k}$, for every $\mu\geq \mu_0$.  Here $H^{r,k}$ takes the place
of $H$ in \S\ref{SCIntro}, or \autoref{thm:ViehwegStability}, and $k>0$ is the index
of the canonical embedding of the KSB-stable varieties in $\sM^{\KSBm}_{d,\nu}$ into
$\bP^{r-1}$ (see \S\ref{ss:Hilbert} for the details in a more general setting).

\subsection{Optimal choices of linearizations}
One would ideally like to find a stable linearization over the subscheme $H^{r,k}$
(as in \S.\ref{higher}) to be precisely
$\det\big( f_*\omega^{[k]}_{\mathfrak X/H^{r,k}}\big)$, cf.~\ref{eq:KSBlin}.  However
ampleness of $\det \big(f_*\omega^{[k]}_{\mathfrak X/H^{r,k}}\big)$ (or equivalently
the CM line bundle) beyond the locus of $H^{r,k}$ parametrizing smooth canonically
polarized projective varieties or those with only klt singularities, 
established in \cite{Viehweg95}, \cite{Viehweg06}, is not known.  Our
approach in this paper only shows that a \emph{product} of such line bundles as in 
\autoref{eq:KSBlin} is ample over the entire $H^{r,k}$
(see \autoref{thm:ViehwegStability}).

\subsection{Acknowledgements.}
We are thankful to Kenneth Ascher, Harold Blum, Jack Hall, Giovanni Inchiostro,
J\'anos Koll\'ar, and Julius Ross for answering our questions and for useful
comments.

\section{Preliminaries and background}
\label{s:background}

For the basic definitions and properties of singularities, including those of
\emph{semi-log-canonical singularities}, or \emph{slc} for short, we refer the reader
to \cite{KollarSingsOfTheMMP}.  We will generally follow \cite{ModBook} for notions
of (KSB-)stability of varieties and their moduli problems.

\begin{definition}\label{def:families-pairs}
  In this article every scheme will be assumed to be defined over $\bC$ and a
  \emph{variety} means a reduced, finite type, but not necessarily irreducible scheme
  (over $\bC$).
  %
  %
  \end{definition}

\begin{notation}\label{not:det}
  Given a reduced scheme $X$ and a coherent sheaf $\sF$ of rank $r$, for any
  $m\in \bN$, we define $\sF^{[m]} : = \big( \sF^{\otimes m} \big)^{**}$, where
  $(\sblank)^{**}$ denotes the double dual. 
  We will apply the same notation
  for all tensor operations. In particular,
  $\bigwedge^{[m]}\sF\leteq \left(\bigwedge^m\sF\right)^{**}$.  Furthermore,
  $\det\sF$ will denote $\bigwedge^{[r]}\sF=\big(\bigwedge^r \sF\big)^{**}$. Notice
  that if $X$ is regular, then $\det\sF$ is a line bundle.  
\end{notation}

\begin{definition}\label{def:divisorial}
We call a coherent sheaf $\sL$ of rank 
one a \emph{divisorial sheaf}, if it is $S_2$ and locally free in codimension one \cite[Def.~3.25]{ModBook}.
\end{definition}

\subsection{Weakly positive sheaves}\label{s:WP}
We briefly review some key positivity notions for coherent sheaves that are relevant
for our purposes, while referring to \cite[\S.2.3]{Viehweg95} for an in-depth
treatment of this material and further background.

\begin{definition}
  A locally free sheaf $\sE$ on a proper scheme $Y$ is said to be \emph{nef}, if for
  every proper nonsingular curve $C$ and any morphism $\gamma: C\to Y$,
  $\deg(\sL) \geq 0$ for every invertible quotient $\sL$ of $\gamma^*\sE$.
\end{definition}

\begin{definition}\label{def:GGG}
  An $\sO_Y$-module $\sF$ on a reduced scheme $Y$ is called \emph{globally generated
    over an open subset $V\subseteq Y$}, if the natural map 
  \[
    H^0(Y, \sF) \otimes_{\bC} \sO_{V} \longrightarrow \sF \otimes_{\sO_Y} \sO_{V}
  \]
  is surjective over $V$. 
\end{definition}

\begin{defini}[\protect{\cite[Def.~2.11]{Viehweg95}}]
\label{def:WP}
Let $Y$ be a quasi-projective variety and $U\subseteq Y$ an open subset.  A torsion
free sheaf $\sF$ on $Y$ is called \emph{weakly positive over $U$}, if $\sF\resto U$
is locally free, and for some ample line bundle $\mathcal H$ on $Y$, for every
$\alpha\in \bN$, there exists a $\beta\in \bN$ such that
  \[
    \sym^{\alpha \cdot k}\sF \otimes \mathcal H^{k}
  \]
  is globally generated over $U$, for any multiple $k$ of $\beta$.  Note that this
  definition is independent of the choice of
  $\mathcal H$~\cite[Lem.~2.14]{Viehweg95}.  A torsion free sheaf $\sG$ on $Y$ is
  \emph{weakly negative} (over $U$), if its dual $\sG^*$ is weakly positive (over
  $U$).
\end{defini}

\begin{remark}[\protect{\cite[2.12.2]{Viehweg95}}]
  \label{rem:nef}
  If $Y$ is projective and $\sF$ is locally free, 
  then $\sF$ is weakly positive over $Y$, if and only if it is nef.
\end{remark} 

\begin{proposition}[\protect{Lem.~\cite[2.15.(2)]{Viehweg95}}]\label{prop:Trace}
  Let $\sigma: X\to Y$ be a finite morphism of reduced quasi-projective schemes and
  $\sF$ a torsion free sheaf on $Y$.  
 \begin{enumerate} 
 \item\label{item:splittrace} Assume that the trace map splits the natural map
   $\sO_Y \into \sigma_*\sO_X$.  With $U:= \gamma^{-1}(V)$, we have:
   \[
     \text{$\gamma^* \sF$ is weakly positive over $U$ $\iff$ $\sF$ is weakly positive
       over $V$.}
   \]
 \item\label{item:splitflat} Assume that $\sigma$ is flat. Then the above splitting
   condition holds and hence so does the conclusion in 
   \autoref{item:splittrace}
   \cite[\href{https://stacks.math.columbia.edu/tag/0BVH}{Tag 0BVH}]{stacks-project}.
  \end{enumerate}
\end{proposition}

\subsection{Families of varieties}\label{FAMS}
\begin{defnot}\label{def:canonical}
  For a morphism of finite type $f: X\to B$ of relative dimension $n$, we set
  $\dcx{X/B}$ to denote the relative normalized dualizing complex $f^! \sO_B$ over
  $B$ and define the relative canonical sheaf by
  $\omega_{X/B} : = h^{-n}\big(\dcx{X/B})$.  For an $S_2$ and $G_1$ (Gorenstein in
  codimension one) scheme $X$, a canonical divisor is denoted by $K_X$ (see
  \cite[\S5]{Kov13} for more details).
  \end{defnot}
  \begin{defnot}\label{ff}
    Given a morphism of finite type $f:X\to B$, let $\sL$ be a \emph{family of
      coherent sheaves of rank one over $B$}. 
    Let $j:U\hookrightarrow X$ be the inclusion map for the locus over which $\sL$ is
    locally free and set $Z:= X\setminus U$. Assume that
    $\codim_{X_b}(X_b\cap Z)\geq 2$, for every $b\in B$. By an abuse of notation, for
    any $m\in \bZ$, define $\sL^{[m]}:= j_*(\sL\resto U)^m$, cf.~\cite[3.35.1]{ModBook}.
  \end{defnot}
  \begin{defnot}\label{ff2}
    Assume that $X_b$ is $G_1$, for every $b\in B$ and let $K_{X/B}$ denote the
    divisor class corresponding to $\omega_{X/B}$.  Let $j:U\hookrightarrow X$ be the
    inclusion map for the locus over which $K_{X/B}$ is Cartier and set
    $Z:= X\setminus U$.  Assume that $\codim_{X_b}(X_b\cap Z)\geq 2$ for each
    $b\in B$, similarly to \autoref{ff}, define
    \begin{equation}\label{eq:Can}
      \omega_{X/B}^{[m]}=  j_*(\omega^m_{U/B}).
    \end{equation}
    This is isomorphic to $\sO_{X}(m K_{X/B})$ \cite[Def.~3.35]{ModBook}.
  \end{defnot}
\begin{remark}
  If $X$ is $S_2$ and $G_1$ and $B$ is Gorenstein, then
  $\omega_{X/B} \simeq \sO_X(K_X - f^*K_B)$ by \cite[(3.3.6)]{Conrad00}. In
  particular $\omega_{X/B}$ is reflexive and of rank one, and $\omega^{[m]}_{X/B}$ in
  \autoref{eq:Can} coincides with the $m$-th reflexive power of $\omega_{X/B}$.
\end{remark}

\begin{notation}[Pullback]
  \label{Pull}
  For any morphism of schemes $f: X\to B$ and $\psi: B' \to B$, $X_{B'}$ denotes
  $X \times_{B} B'$.  Let $\psi': X_{B'}\to X$ and $f': X_{B'}\to B'$ denote the
  natural projections.  Given an $\sO_X$-module $\sF$ on $X$, $\sF_{B'}$ (resp.\
  $(X,\sF)_{B'}$) denotes $(\psi')^*\sF$ (resp.\ $(X_{B'}, (\psi')^*\sF)$).
\end{notation}

\subsection{KSB-stable families}\label{ss:locstable}
\label{KSBFAMS}

\begin{definition}
  A flat, projective morphism of schemes, $f:X\to B$, with reduced, connected,
  positive dimensional fibers, will be called a \emph{flat projective family} or, to
  be concise, simply a \emph{projective family}. A coherent sheaf $\sL$ on $X$ is
  called a \emph{flat family of divisorial sheaves}, if $\sL$ is flat over $B$ and
  $\sL_b$ is divisorial (\autoref{def:divisorial}), for every $b\in B$
  \cite[\S.3.3]{ModBook}.
\end{definition}

\begin{definition}\label{def:LS}\cite[3.40]{ModBook}
  Given two schemes $X$ and $B$, let $f:X\to B$ be a flat morphism of finite type.
  With $U\subseteq X$ as in \autoref{ff2}, $f$ is called \emph{\lst}, if $X_b$ is slc
  for every $b\in B$ and $\omega^{[m]}_{X/B}$ is a flat family of divisorial sheaves
  over $B$ for every $m\in \bZ$.  If in addition it is a proper family with
  $\omega_{X/B}$ $f$-ample, then $f$ is said to be \emph{\ksb}.  When $B=\Spec(\bC)$,
  $X$ is called a stable or a \emph{KSB-stable variety}.
\end{definition}
 \begin{remark}\label{rem:BCStable}
   In a \lst family $f:X\to B$, $\omega^{[m]}_{X/B}$ commutes with arbitrary base
   change \cite[\S 3]{ModBook}.  Assuming that $B$ is of finite type, by using base
   change along the fibres of $f$, we find that, for a sufficiently divisible $m_0$,
   $\omega^{[m]}_{X/B}$ is invertible, for any $m$ divisible by $m_0$.
   \end{remark}
\begin{remark}\label{rem:S2}
  In the setting of \autoref{def:LS},
  if $B$ is $S_2$, then so are $\sO_X$ and $\omega_{X/B}^{[m]}$, for any $m\in \bZ$,
  \cite[\S3]{Hassett-Kovacs04}.  
\end{remark}

\subsection{Gabber's Extension Theorem}\label{GAB}
\begin{set-up}\label{setup:Gab}
  Given a reduced quasi-projective scheme $W$, let $\pi: W'\to W$ be a
  desingularization.  Assume that $\wtilde W$ is a nonsingular projective scheme
  containing $W'$ as a dense open subscheme and let $\sF$ and $\wtilde \sF$ be two
  locally free sheaves on $W$ and $\wtilde W$ respectively, with
  $\wtilde \sF\resto {W'}\simeq \pi^*\sF$.
\end{set-up}

\begin{theorem}[Gabber's Extension Theorem \protect{\cite[\S.5.1]{Viehweg95}}]
  \label{thm:Gab}
  In the setting of \autoref{setup:Gab}, further assume that for any nonsingular
  curve $C$, and an open dense subset $C^\circ\subseteq C$, with a proper morphism
  $\eta^\circ: C^\circ \to W$, there exists a locally free sheaf $\sG_C$ on $C$
  satisfying the following properties.
  \begin{enumerate}
  \item ${\sG_C}\resto {C^\circ} \simeq (\eta^\circ)^*\sF$.
  \item\label{eq:lift} If for each nonsingular curve $C'$, finite morphism
    $\gamma: C'\to C$, and proper morphism $\eta': C'\to \wtilde W$, $\eta'$ is a
    lifting of $\eta^\circ$ in the sense that:
    \begin{equation*}%
      \pi\circ \eta'\resto {\gamma^{-1}(C^\circ)} = \eta^\circ \circ \gamma\resto {\gamma^{-1}(C^\circ)},
    \end{equation*}
    then 
    $(\eta')^* \wtilde \sF \simeq \gamma^* \sG_C$.    
  \end{enumerate}
  Then, there is a projective compactification $\widehat W$ of $W$ and a locally free
  sheaf $\widehat \sF$ on $\widehat W$, with $\widehat \sF\resto W \simeq \sF$, such that
  for any nonsingular proper scheme $W''$ and each pair of birational morphisms
  $\psi: W''\to \wtilde W$ and $\rho: W''\to \widehat W$ removing the indeterminacy
  of $\pi: \wtilde W \dashrightarrow \widehat W$, we have
  $\rho^*\widehat \sF \simeq \psi^*\wtilde \sF$.  In particular if $\wtilde \sF$ is
  nef, then so is $\widehat \sF$.
\end{theorem}

\subsection{Kawamata coverings.}
\begin{deftheo}[\protect{\cite[Prop.~4.1.12]{Laz04-I}, \cite[2.5]{Viehweg95}}]
\label{def:KawCov}
Let $W$ be a nonsingular quasi-projective scheme and $\sum_{k=1}^m\Gamma_k$
a reduced divisor with simple normal crossing (snc) support. For every positive integer $N_k$
there is a nonsingular variety $W'$ together with a finite and flat morphism 
$\tau:W'\to W$ whose ramification index $l_k$
 over each $\Gamma_k$ is any desired multiple of $N_k$. 
 We refer to $\tau$ as a \emph{Kawamata covering}. 
\end{deftheo}

\begin{remark}\label{rem:Kaw}
  Any Kawamata covering consists of a Bloch--Gieseker covering
  \cite[Thm.~4.1.10]{Laz04-I} and a sequence of cyclic root coverings
  \cite[Prop.~4.1.6]{Laz04-I}.  The latter coverings satisfy certain functorial
  properties in the following sense.  Given any quasi-projective variety $W$ and a
  Cartier divisor $D\subset W$, for $m_i>0$, $i=1,2$, let $p_i: Y_i \to W$ be a
  finite and flat morphism of quasi-projective varieties equipped with a line bundle
  $\sL_i$ such that
  \begin{equation}\label{eq:BG}
  \sL_i^{m_i}  \simeq  p_i^*  \sO_{W} (D)  .
  \end{equation}
  With $0\neq s_i \in H^0( Y_i, \sL^{m_i} )$, corresponding to $D$ via
  \autoref{eq:BG}, define
  \[
    W_i := \Spec_{{Y_i}} \big(
    \bigoplus_{j=0}^{m_i-1} \sL_i^{-j} \big)
    \overset{q_i}{\longrightarrow} Y_i
  \]
  where $\bigoplus_{j=0}^{m_i-1} \sL_i^{-j}$ is made into a sheaf of algebras via
  multiplication by $s_i$ \cite[p.65]{KM97}, \cite[Ex.~5.17.(c)]{Ha77}.  Let
  $\tau_i= p_i \circ q_i : W_i \to W$ be the induced finite and flat morphism as in
  \autoref{def:KawCov}.  Assume that there is a morphism $\phi': Y_2 \to Y_1$, such
  that $p_2= p_1\circ \phi'$.  Now, if $m_2$ is divisible by $m_1$, then by the
  functoriality of $\Spec_{{Y_i}}(\cdot)$, and using $\phi'$, $\tau_2: W_2\to W$
  lifts to a morphism $\phi: W_2\to W_1$, such that $\tau_1\circ\phi= \tau_2$.
\end{remark}

\subsection{Stratified coverings.}\label{ss:covering}
In this section, we modify Viehweg's covering construction \cite[\S
5.2]{Viehweg95}---adapted to a filtration by closed subschemes and prescribed
ramification indices along designated divisors---so that it applies to the situation
we analyze in \S \ref{s:stratified}.

\begin{set-up}\label{back}
  Let $B$ be a reduced quasi-projecitve scheme and $D\subset B$ a closed subscheme,
  not dense in any of the components of $B$.  Let $t\in\bN$ and for each
  $i\in\{1,\dots,t\}$, let $\overline \Yi\subseteq B$ be a closed subscheme that is
  not contained in $D$, and that does not contain any component of $D$ or $B$, unless
  $\overline \Yi=B$.  Further let $\pii: \Wi\to \overline \Yi$ be a projective
  birational morphism with $\Wi$ nonsingular such that
  \begin{equation}\label{BIR}
    \pii ^{-1} \big( \overline \Yi \cap D\big) =
    \sum_{k=1}^{d_i} \Gamma^{(i)}_k
   \end{equation} 
   is a reduced divisor with simple normal crossing support.  Given
   $N_k^{(i)}\in \bN$, $1\leq k \leq d_i$, assume that $\Wi$ is equipped with a
   finite morphism of nonsingular quasi-projective schemes
   \begin{equation}\label{FINITE}
  \taui: {\Wi}'\to \Wi
  \end{equation}
  whose ramification index over each $\Gamma^{(i)}_k$ is divisible by $N_k^{(i)}$.
  Such a morphism is ensured to exist by \autoref{def:KawCov}.
\end{set-up}

\begin{lemma}\label{lem:covering}
  In the situation of \autoref{back}, for each $i$ there exist
  \begin{enumerate}
  \item a closed subscheme $\Lambdai{i+1}\subset B$, such that
    $(\Lambdai{i+1}\setminus D) \subset \overline \Yi$,
  \item a finite and flat morphism $\widehat\sigma_i : \widehat V_i\to B$,
    and
  \item\label{eq:factorsgeneric} a rational map
    $\widehat \etai: Z_i:= \widehat \sigma_i^{-1}(\overline \Yi) \dashrightarrow
    {\Wi}'$ such that
    \begin{equation*}      
      \pii \circ \taui\circ \widehat\etai\resto {Z_i\setminus
        \widehat\sigma_i^{-1}(\Lambdai{i+1})}=  \widehat{\sigma_i}\resto {Z_i\setminus
        \widehat\sigma_i^{-1}(\Lambdai{i+1})}.
    \end{equation*}
  \end{enumerate}
\end{lemma}

\begin{proof}
  Without loss of generality we may identify $B$ with its image under an embedding
  into a projective space $\bP^a$. Furthermore,
  $(\pii \circ \taui): {\Wi}' \to \overline \Yi \subset \bP^a$ is such that each
  irreducible component is generically finite over its image, it follows from
  \cite[Claim.~5.8]{Viehweg95} that there is a finite and flat morphism of projective
  varieties $\overline \sigma: P\to \bP^a$ and a rational map
  $\widehat\etai: Z_i\dashrightarrow {\Wi}'$ satisfying \autoref{eq:factorsgeneric},
  where $\widehat V_i:= \overline \sigma^{-1}(B)$,
  $\widehat\sigma_i: = \overline \sigma\resto {\widehat V_i}$, and $\Lambdai{i+1}$ is the
  union of $D$ with the image of the indeterminacy locus of $\widehat\etai$ under
  $\widehat \sigma_i$.
\end{proof}

\begin{proposition}\label{prop:VStrat}
  In the setting of \autoref{back} 
  there exist a strictly
  decreasing filtration by reduced closed subschemes
  \[
    B = \Lambdai{1} \supsetneq \Lambdai{2} \supsetneq \ldots \supsetneq
    \Lambdai{t} \supsetneq \Lambdai{t+1} = D,
  \]
  for some $t\in \bN$,
  and a finite and flat morphism $\sigma: V\to B$ of
  reduced quasi-projective schemes
  satisfying the following properties. For each $1\leq i \leq t$,
  \begin{enumerate}
  \item\label{item:1} $\overline Y_i$, the closure of
    $\Yi:= \Lambdai{i}\setminus \Lambdai{i+1}$, in $\Lambdai{i}$, coincides with
    $\overline{\Lambdai{i}\setminus D}$.
  \item\label{item:lifts} Let $\pii: \Wi\to \overline \Yi$ and $\taui: {\Wi}'\to \Wi$
    be any morphisms satisfying the conditions in \autoref{back}. Then there is a
    morphism $\etai: 
    \sigma^{-1}(\Yi) \to {\Wi}'$ factoring $\sigma\resto {\sigma^{-1}(\Yi)}$ via
    $\pii\circ \taui: {\Wi}'\to \overline\Yi$:
    \[
      \xymatrix{
        &&&  {\Wi}' \ar[d]^(.4){\taui} \\
        &&&  \Wi\ar[d]^(.35){\pii} \\
        V \ar@/_5mm/[rrrrrl]_{\sigma} & \sigma^{-1}(\Yi) \ar[rruu]^{\etai}
        \ar[l]_{\supseteq} \ar[rr]^{\sigma\resto {\sigma^{-1}(\Yi)}} && \overline
          \Yi\ar[r]^{\subseteq} & B .}
    \]
  \item\label{item:3} The collection of locally closed subschemes $\{ \Yi \}_{i=1}^t$
    forms a stratification of $B\setminus D$, that is, $B\setminus D =\bigcup_i \Yi$,
    and $Y_j\cap Y_{j'}=\emptyset$ for $j\neq j'$.
  \end{enumerate}
\end{proposition}

\begin{proof}
  Let $\Lambdai{1}= B$, $\overline Y_1\leteq B$, and take $\pi_1:\Wii 1\to B$, and
  $\tau_1:{\Wii 1}'\to\Wii 1$ to be as in \autoref{back}.  
  Then \autoref{lem:covering} gives a finite and flat morphism $\sigma_1: V_1\to B$
  of quasi-projective schemes and a rational map
  $\ol\eta_1: V_1\dashrightarrow\Wii 1'$.  Let $\Lambdai{2}$ be the union of $D$ and
  the image of the indeterminacy locus of $\ol\eta_1$ under $\sigma_1$ and let
  $Y_1\leteq B\setminus\Lambdai 2$. As $\Lambdai 2$ is nowhere dense in $B$, it
  follows that $\overline Y_1=B$ is the closure of $Y_1$. Furthermore, as
  $\Lambdai 2$ contains the image of the the indeterminacy locus of $\ol\eta_1$, it
  follows that $\ol\eta_1$ is a morphism on $\sigma_1^{-1}(Y_1)$.

  Next, we define $\overline Y_2\leteq\overline{\Lambdai 2\setminus D}$, and let
  $\pi_2:\Wii 2\to\overline Y_2$, and $\tau_2:{\Wii 2}'\to\Wii 2$ be as in
  \autoref{back}.
 We use
  \autoref{lem:covering} again, this time for $\overline Y_2$, and find $\widehat \sigma_2: \widehat V_2 \to B$. 
  Then, through base change to $V_1$, 
   we obtain $\sigma_2:V_2\to V_1$, a rational map
  $\ol\eta_2: V_2\dashrightarrow \Wii 2'$, and define $\Lambdai{3}$, as above, to be
  the union of $D$ and the image of the indeterminacy locus of $\ol\eta_2$ under
  $\sigma_1\circ\sigma_2$. We then define
  $Y_2\leteq\Lambdai 2\setminus\Lambdai 3$. By construction $\Lambdai 3$ is nowhere
  dense in $\Lambdai 2$ and hence $\overline Y_2$ is the closure of $Y_2$. Furthermore, 
  as above, $\ol\eta_2$ defines a morphism on $(\sigma_2\circ\sigma_1)^{-1}(Y_2)$.

  By a repeated application of this construction,
  we thus find:
 \begin{enumerate}
 \item a decreasing filtration $\Lambdai{\kdot}\subseteq B$ consisting of $(t+1)$
   steps,
 \item 
   a finite flat morphism $\sigma_i:V_i\to V_{i-1}$, $2\leq i\leq t$, equipped with a
   morphism
    \[
      \ol\etai\resto{(\sigma_i \circ\ldots \circ \sigma_1)^{-1} (Y_i)} : \big(
      \sigma_i \circ\ldots \circ \sigma_1\big)^{-1} (Y_i) \longrightarrow {\Wi}',
    \]
  \item a reduced quasi-projective scheme $V:= V_t$ and a finite flat morphism
  $\sigma: V\to B$ defined by the composition: 
  \[
    \xymatrix{ V= V_t \ar@/_7mm/[rrrr]^{\sigma} \ar[r]^-{\sigma_t} & \ldots \ar[r] &
      V_2 \ar[r]^{\sigma_2} & V_1 \ar[r]^-{\sigma_1} & B .  }
  \]
  \end{enumerate}
  We now define $\eta_t= {\ol\eta_t}\resto{\sigma^{-1}(Y_t)}$, and for
  $1\leq i\leq t-1$:
  \[
    \etai:= \ol\etai \circ \big( \sigma_{i+1} \circ \ldots \circ \sigma_t
    \big)\resto{\sigma^{-1}(\Yi)}.  
  \]
 From this construction it is now straightforward to see that the properties \autoref{item:1}, \autoref{item:lifts} and \autoref{item:3} hold. 
\end{proof}

\subsection{Moduli functors of polarized varieties}
\label{ss:Good}

For a detailed background on moduli problems of projective varieties we refer to
\cite{ModBook}, \cite{Kollar90} and \cite{Viehweg95}. An extensive treatment of the
theory of stacks can be found in \cite{LMB00} and \cite{Olsson16}.

\begin{definition}\label{def:mod}
  A $\bQ$-ample line bundle is a divisorial sheaf such that $\sL^{[r]}$ is an ample
  line bundle, for some $r\in \bN$.
\end{definition}

The moduli problem $\sF$ of polarized $S_2$-schemes. 
(cf.~\cite[\S.1.1]{Viehweg95}, \cite[\S.2.1]{Kollar90}) consists of \emph{objects}:
\begin{equation*}
  \begin{multlined}
    \sF({\Spec \bC}) = \{ [(Y,L)] \; \bigm\mid \text{$Y$ is projective,
      connected, reduced $S_2$-scheme,} \\
    \text{and a $\bQ$-ample line bundle $L$}\} ,
  \end{multlined}
\end{equation*}
where $[(Y,L)]$ denotes the isomorphism class of the pair $(Y,L)$, and
a functor $\sF : \mathfrak{Sch}_{\bC} \to \mathfrak{Sets}$
defined by
%
%
%
\begin{equation*}
\begin{multlined}
  \sF (V) = \Big\{ 
  (f: U\to V, \sL)\; \bigm\mid
  \text{$f$ is flat and projective, $\sL$ is of rank one,} \\
  \qquad\qquad\text{$\sL^{[m]}$ is a flat family of divisorial sheaves over $B$ for
    all $m\in\bZ$,}
  \\
  \left.  \text{ and $(U,\sL^{[1]})_v\in \sF({\Spec \bC})$, for all $v\in V$} \Big\}
  \right/ \sim ,
\end{multlined}
\end{equation*}
where the equivalence relation $\sim$ is given by 
\begin{equation}\label{equivalence}
\begin{aligned}
  \big( f_1: U_1 \to V , \sL_1 \big) \sim \big( f_2: U_2 \to V , \sL_2 \big) \iff &
  \ \parbox[c]{15em}{ there is a $V$-isomorphism $\sigma: U_1\to U_2$ such that
    $\sL_1 \simeq \sigma^* \sL_2 \otimes f_1^*\sB$, for some line bundle $\sB$ on
    $V$.}
\end{aligned}\hskip-2em\smallskip
\end{equation}
Sometimes we refer to the objects of this moduli problem using $\Ob(\sF)$.

We say that a subfunctor $\sM \subseteq \sF$ is a \emph{moduli functor}, if it is
closed under arbitrary base-change, i.e., for every morphism of schemes $V' \to V$ we
have
\[
\big( (U, \sL) \to V  \big) \in \sF(V) \implies  (U, \sL)_{V'} \in \sF(V')  .
\]
\begin{definition}\label{def:functoial}
A moduli functor $\sM$ is said to be \emph{functorially polarized}, if 
$\sB\simeq \sO_B$ in \autoref{equivalence} ~\cite[Def.~1.3]{Viehweg95} 
and \cite[Def.~2.3]{Kollar90}.
\end{definition}
\begin{example}\label{KSB}
  The notion of a functorially polarized moduli functor is modeled on the
  \emph{\ksb-moduli functor of stable varieties} defined by
  \begin{equation*}
    \sM^{\KSBm} (V)  =   \Big\{ (f: U\to V, \omega_{U/V})\;  \bigm\mid  \text{$f$ is
      \ksb (\autoref{def:LS}} \Big\} / \sim , 
  \end{equation*}
  Here the equivalence relation $\sim$ is canonically given by
  (cf.~\cite[7.9.3]{KK10}),
  \[
    \begin{aligned}
      \big(  f_1: U_1 \to V  \big)  \sim  \big( f_2: U_2 \to V    \big)   \iff &
     \text{there is a $V$-isomorphism $\sigma: U_1\to U_2$}  . 
    \end{aligned}
  \]
\end{example}
\begin{remark}\label{rem:invertible}
  Given a moduli functor $\sM$, if there is an $m_0\in \bZ$ such that $L^{[m_0]}$ is
  invertible for every $(Y,L)\in \Ob(\sM)$, then using base change over closed points
  of $B$, $\sL^{[m]}$ is invertible for any $\big( f: (X,\sL)\to B \big) \in \sM(B)$
  for any $m$ divisible by $m_0$.
\end{remark}
\begin{terminology}
  Sometimes we refer to $\sL$ in $(f: U \to V, \sL) \in \mathfrak M(V)$ as a
  \emph{polarization} for $f$.
\end{terminology}
\begin{definition}[\protect{\cite[Def.~10.83]{ModBook}}]
  We say that a morphism of schemes $j: B'\to B$ is a \emph{locally closed partial
    decomposition}, if the restriction of $j$ to every connected component of $B'$ is
  a locally closed embedding.
\end{definition}
\begin{definition}
\label{def:openclosed}
Let $\sF$ be as in \autoref{def:mod}.
We say that a submoduli functor $\sG\subset \sF$ is \emph{locally closed}, resp.\ a
\emph{locally closed partial decomposition}, if for every
$(f: X \to B, \sL) \in \sG(B)$, there is a (universal) subscheme $B^u\subseteq B$
such that the inclusion map $j:B^u \hookrightarrow B$ is a locally closed embedding,
resp.\ a locally closed partial decomposition, and such that for every morphism
$\phi: T\to B$ we have: $(X_T \to T, \sL_T) \in \sG(T)$, if and only if there is a
factorization\vskip-2.5em
\[
  \xymatrix{ T \ar@/_5mm/[rr]^{\phi} \ar[r] & B^u \ar[r]^{j} & B.  }
\]

\noindent
We say that $\sG$ is \emph{constructible} if $B^u$ is such.
\end{definition}
\begin{notation}
  Given a moduli functor of polarized schemes $\sF$, by $\sF_h$ and
  $\sF_{\nu}$ we denote the submoduli functors whose objects $(Y,L)$ have
  $h\in \bQ[x]$, and respectively $\nu \in \bQ$, as their Hilbert polynomial and
  volume, with respect to $L$.
\end{notation}
\begin{definition}\label{def:bounded}
  We say that a moduli functor $\sM$ is bounded, if there is a $k_0\in \bN$ such
  that, for every $(Y, L)\in \Ob(\sM)$, $L^{[k_0]}$ is a strongly ample invertible sheaf, 
  that is, it is very ample and that $h^i(Y, L^{[k_0]})=0$, for all $i>0$ \cite[Def.~8.38]{ModBook}.
\end{definition}
\begin{remark}\label{rem:KSB}
  According to \cite{HMX13}, for every $d, \nu>0$, $\sM^{\KSBm}_{d, \nu}$ is bounded
  and that $\sM_{d, \nu}^{\KSBm}= \bigsqcup_{h_i\in I} \sM^{\KSBm}_{h_i}$, for some
  $I=\{ h_i \in \bQ[x] \}$ with $| I |<\infty$.  Furthermore, by
  \cite[Cor.~8.52]{ModBook}, for any $h\in \bQ[x]$, $\sM^{\KSBm}_h$ is constructible.
\end{remark}
\begin{remark}\label{rem:ABC}
  Given a moduli functor $\sM$ and $\big( f:(X,\sL) \to B \big)\in \sM(B)$, by using
  base change to closed points of $B$, we find that the function
  $b\mapsto \chi(X_b, \sL_b^{[k]})$ is locally constant, for every $k\in \bZ$.  If
  $\sM$ is also assumed to be bounded, then with $k_0>0$ as in \autoref{def:bounded},
  for every $k$ divisible by $k_0$ the function $b\mapsto h^0(X_b, \sL_b^{[k]})$ is
  also locally constant.  Therefore, if $B$ is Noetherian, then $f_*\sL^{[k]}$ is
  locally free \cite[Thm.~12.11]{Ha66} and it commutes with arbitrary base change for
  every $k$ for which $\sL^{[k]}$ is invertible.
\end{remark}
%

\section{Stratified extension and weak positivity over arbitrary reduced
  quasi-projective schemes}\label{s:stratified}
  
The primary aim of this section is proving \autoref{thm:WPkey}. Stratified extensions
(\S.\ref{ss:extension} below) is the key technical tool for our proof.

\subsection{Stratified extensions for proper DM stacks}\label{ss:extension}
\begin{definition}[Extension]\label{def:reduction}
  Let $\sM$ be a moduli functor of polarized schemes.  Given a scheme $W$ and a
  reduced closed subscheme $\Gamma\subset W$, not dense in any of the component of
  $W$, let
  \begin{equation}\label{reduct}
    \big( f: (X,\sL) \to W \setminus \Gamma  \big)  \in \sM(W\setminus \Gamma) .
  \end{equation}
  Following \cite[p.11]{ModBook}, we say that a generically finite and surjective
  morphism of schemes, $\tau: W'\to W$, and a family $g:(X', \sL') \to W'$ define an
  \emph{extension} of \autoref{reduct} in $\sM$, if
  $\big(g:(X', \sL') \to W'\big) \in \sM(W')$, and 
  \[
    (X',\sL')_{W'\setminus \tau^{(-1)}(\Gamma)} \simeq (X, \sL)\resto {W'\setminus
      \tau^{(-1)}(\Gamma)}.
  \]
\end{definition}
\begin{lemma}\label{lem:stratified}
  Let $\sM$ be a separated moduli functor with a proper DM stack $\mathcal M$.  Given
  a reduced quasi-projective scheme $B$ and a reduced closed subscheme $D\subset B$,
  not dense in any of the components of $B$, let
  \begin{equation}\label{eq:red2}
    \big(  f:(X, \sL) \to B\setminus D \big)\in \sM(B\setminus D) .
  \end{equation}
  Then, there are
  \begin{enumerate}
  \item a strictly decreasing filtration by reduced closed subschemes
    $B = \Lambdai{1}\supsetneq \Lambdai{2} \supsetneq \ldots \supsetneq \Lambdai{t}
    \supsetneq \Lambdai{t+1} = D$,
    and an induced stratification of $B\setminus D$ by locally closed subschemes
    $\{ \Yi \; | \; 1\leq i\leq t \}$, where
    $\Yi:= \Lambdai{i}\setminus \Lambdai{i+1}$,
  \item\label{item:fam} a desingularization $\pii: \Wi\to \overline \Yi$ of the
    closure $\overline \Yi$ of $\Yi$ in $\Lambdai{i}$, and an extension of
    \autoref{eq:red2} over $\Yi$ in $\sM$ defined by a Kawamata covering of
    nonsingular schemes $\taui: {\Wi}'\to \Wi$ (\autoref{def:KawCov}) and
    $f_i' : (X_i' , \sL_i' ) \to {\Wi}'$,
  \item\label{item:split} a finite and flat morphism $\sigma: V\to B$,
  \item\label{item:facts} a morphism $\etai:\sigma^{-1}(\Yi)\to {\Wi}'$ factoring
    $\sigma\resto{\sigma^{-1}(\Yi)}$, that is
    $( \pii \circ\taui) \circ \etai = \sigma\resto{\sigma^{-1}(\Yi)}$, and
  \item\label{item:bigred} a nonsingular quasi-projective scheme $\wtilde V$ and a
    morphism $\wtilde\eta_1: \wtilde V\to\Wii 1'$ resolving the indeterminacy of
    $\eta_1: V \ratmap \Wii 1'$ via the projective birational morphism
    $\pi:\wtilde V\to V$.
  \end{enumerate}
  In particular the pullback $\big( X_1', \sL_1'\big)_{\wtilde V}$ via
  $\wtilde \eta_1$ is an extension of $(X,\sL)_{V^\circ}\to V^\circ$, where
  $V^\circ:= V\setminus \sigma^{-1}(D)$ (see \autoref{sumnot}).
\end{lemma}

\begin{proof} Let $\mcM$ be the proper algebraic coarse moduli space of $\mathcal M$
  \cite{Kollar97a}, \cite{KM97}, \cite[Thm.~6.12]{Ryd13}.  Our aim is to construct
  the filtration $\Lambdai{\kdot}\subseteq B$ as in \autoref{prop:VStrat} using
  particular morphisms $\pii$ and $\taui$ satisfying the required properties.
  
  Let $\overline \Yi\subseteq B$ be any closed subscheme as in 
  \autoref{back} and take 
  $\mu_i: \overline \Yi\dashrightarrow \mcM$ to be the moduli map
  induced by $(X, \sL)\resto {\overline \Yi}$. After replacing $\mcM$ by  
  a nonsingular model, let $\pii: \Wi\to \overline \Yi$ be a 
  desingularization that (i) resolves the indeterminacy locus of $\mu_i$ 
  via $\mu^{(i)}: \Wi\to \mcM$, 
  (ii) $\sum_{k=1}^{d_i} \Gamma^{(i)}_k$ as in
  \autoref{BIR} is snc, and such that, (iii) for each $k$,
  $\Delta^{(i)}_k:=\mu^{(i)}(\Gamma^{(i)}_k)$ is a divisor in $\mcM$.

  As a separated DM stack, $\mathcal M$ is \'etale-locally isomorphic to a quotient
  stack~\cite[Lem.~2.2.3]{AV02}, that is, there is an \'etale covering
  $\{ p_{\alpha}:V_{\alpha}\to \mcM\}$ of $\mcM$ by affine schemes such that for each
  $\alpha$ there is an affine scheme $U_{\alpha}$ and a finite group $G_{\alpha}$
  acting on $U_{\alpha}$ such that the pullback of the coarse moduli map
  $\mathcal M \to \mcM$ to $V_{\alpha}$ is, up to isomorphism, given by
  $[U_{\alpha}/ G_{\alpha}] \to U_{\alpha}/G_{\alpha}$.
  
  Now, assuming that $\Delta^{(i)}_k\cap V_{\alpha}\neq \emptyset$, let $N_k^{(i)}$
  be the ramification index of the morphism of schemes $U_{\alpha}\to V_{\alpha}$
  over the generic point of $\Delta^{(i)}_k$, $1\leq k \leq d_i$.  Let
  $\taui: {\Wi}'\to \Wi$ be any Kawamata covering as in \autoref{def:KawCov}.  Using
  this recipe for $\pii$ and $\taui$ we apply \autoref{prop:VStrat} to find
  $\Lambdai{\kdot}\subseteq B$.
  
  \begin{claim}\label{claim:lifts}
    With $\overline \Yi$, $1\leq i\leq t$, and $\Lambdai{\kdot}$ as in
    \autoref{prop:VStrat}, and the prescribed $\Wi$ and ${\Wi}'$ as above, we can
    find an open subset of ${\Wi}'$ over which
    $(\mu^{(i)}\circ \taui): {\Wi}'\to \mcM$ lifts to $\mathcal M$ over the generic
    point of each $\Gamma^{(i)}_k$. We denote this lift by $\widehat \mu_{i,k}$.
  \end{claim}
  \begin{proof}[Proof of \autoref{claim:lifts}]
    We may assume with no loss of generality that generically
    $U_{\alpha}\to U_{\alpha}/G_{\alpha}$ is a composition of an \'etale morphism
    with an $N^{(i)}_k$-th root covering (\autoref{rem:Kaw}).  The existence of the
    lift in the \'etale case is straightforward. In the non-\'etale case it is a
    consequence of \autoref{rem:Kaw}.
  \end{proof}
  \noindent
  As $\mathcal M$ is separated, the collection of lifts $\widehat \mu_{i,k}$ to
  $\mathcal M$ in \autoref{claim:lifts} glue and form a lift $\widehat\mu_i$ of
  $(\mu^{(i)}\circ \taui)$ to $\mathcal M$ over an open subset of ${\Wi}'$ whose
  complement is of $\codim_{{\Wi}'}\geq 2$.  Furthermore, as ${\Wi}'$ is nonsingular,
  by \cite[Lem.~2.4.1]{AV02} this lift (uniquely) extends across the entire ${\Wi}'$:
  \[
    \xymatrix{
      && \mathcal M  \ar[dr]  \\
      {\Wi}' \ar[urr]^{\widehat \mu_i} \ar[rrr]^{\mu^{(i)}\circ \taui} &&& \mcM }
  \]
  After pulling back the universal polarized family over $\mathcal M$ via
  $\widehat \mu_i$ we find the desired family $f_i': (X_i', \sL_i') \to {\Wi}'$ in
  \ref{item:fam}, which by construction is an extension of \autoref{eq:red2} over
  $\Yi$ in $\sM$.  Items~\autoref{item:split} and \autoref{item:facts} follow
  directly from \autoref{prop:VStrat}. Item~\autoref{item:bigred} is a consequence of
  \autoref{item:fam} and \autoref{item:facts} when $i=1$.
\end{proof}

\begin{sumnot}\label{sumnot}
  We summarize the constructions and notations introduced in and arising from
  \autoref{lem:stratified} in the following diagram.
  \begin{equation}\label{diag:Big}\hskip-2em
    \begin{aligned}
      \resizebox{\displaywidth}{!}{ \xymatrix{ (\wtilde U, \wtilde \sL):= (X_1',
          \sL_1')_{\wtilde V} \ar[ddd]^{\wtilde g} &&&& & & \Wii{1}' \ar[dd]^{\tau_1}
          & (X_1',          \sL_1') \ar[l]  \\
          &&&& {\Wi}' \ar[d]^{\taui} & (\Xii', \sLi')
          \ar[l]_(.55){\fii'}    \\
          &&&& \Wi \ar[d]^{\pii} & & \Wii 1
          \ar[d]^{\pi_1}      \\
          \widetilde V \ar[r]^{\pi} \ar@/^4em/[rrrrrruuu]^{\wtilde \eta_1} & V
          \ar@/_6mm/[rrrrr]^{\sigma} & \sigma^{-1}(\Yi) \ar[l]_{\supseteq} \ar[rruu]^{\etai}
          \ar[r]^(0.6){\sigma} & \Yi \ar[r]^(0.6){\subseteq} & \overline \Yi
          \ar[rr]^{\subseteq}  &&  B        \\
          (X_{V^\circ} , \sL_{V^\circ} ) \ar[r]^g & V^\circ \ar[u]_{\subseteq}
          \ar[rrrrr]^{\sigma\resto {V^\circ}} &&&&& B\setminus D \ar[u]^{\subseteq} & (X,
          \sL) \ar[l]_{f} }}\hskip-2em
    \end{aligned}
  \end{equation}
\end{sumnot}

\subsection{Proof of \autoref{thm:WPkey}}\label{ss:WPkey}
By compactifying it suffices to show that there is $k_0>0$ such that for every
reduced projective scheme $B$ and a closed subscheme $D\subset B$, not dense in 
any component, 
we have
\[
  \big( f: (X, \sL)\to B\setminus D \big) \in \sM(B\setminus D) \implies f_*\sL^{[k]}
  \; \text{is weakly positive over $B\setminus D$} ,
\]
for every $k>0$ divisible by $k_0$. 

Working in the setting of \autoref{lem:stratified}, we first note that as $\{ \Yi\}$
forms a stratification of $B\setminus D$, so does $\{ \sigma^{-1}(\Yi) \}$ for
$V^\circ= V \setminus \sigma^{-1}(D)$.  As such, given any nonsingular curve $C$ and
an open sense subset $C^\circ \subseteq C$ equipped with a proper morphism
$\eta^\circ: C^\circ \to V^\circ$, there exists $i$, with $1\leq i\leq t$, such that, after
removing a finite number of points from $C^\circ$ if necessary, we have
$\eta^\circ(C^\circ) \subseteq \sigma^{-1}(\Yi) \setminus \sigma^{-1}(\Yii{i+1})$.

Now, with $d_0$ in \autoref{def:SP}, let $k_0> d_0$ be as in \autoref{def:bounded}.
Assuming that $k>0$ is divisible by $k_0$, define
\[
  \sF= g_*\sL^{[k]}_{V^\circ}, \qquad \wtilde \sF = \wtilde g_* \wtilde \sL^{[k]} ,
  \qquad \sG_{C^\circ} = (\eta^\circ)^* \big( g_*\sL_{V^\circ}^{[k]} \big) .
\]
Using the fact that $B$ and therefore $V$ and ${\Wi}'$ are proper, let
$\eta: C\to V$ and $\overline \etai: C\to {\Wi}'$ denote the extensions of
$\eta^\circ: C^\circ \to \sigma^{-1}(\Yi)$ and
$\etai\circ \eta^\circ : C^\circ \to {\Wi}'$, respectively.

Given any finite morphism $\gamma: C'\to C$ of nonsingular curves, assume that
$\eta': C'\to \wtilde V$ is a lifting of $\eta$ in the sense of \autoref{eq:lift}:
\[
\xymatrix{
\wtilde V \ar[rr]^{\pi} &&  V \\
C' \ar[rr]^{\gamma} \ar[u]^{\eta'}  &&  C \ar[u]_{\eta}  .
}
\]
Define $f_C: ((\Xii')_C, (\sLi')_C) \to C$ to be the pullback of
$(\Xii', \sLi') \to {\Wi}'$ via $\overline \etai: C\to {\Wi}'$.
\begin{claim}\label{claim:isoms}
  With $\sG_C:= (f_C)_*\big( (\sLi')_C \big)^{[k]}$, we have
  ${\sG_C}\resto {C^\circ} \simeq (\eta^\circ)^* \sF$ and
  $\gamma^*\sG_C\simeq (\eta')^*\wtilde \sF$.
\end{claim}
\noindent
\begin{proof}[Proof of \autoref{claim:isoms}] As $\fii': (\Xii', \sLi')\to {\Wi}'$
  defines an extension of $(X,\sL)_{\Yi}$, we have
  \begin{equation}\label{eq:rest}
    (g_*\sL^{[k]}_{V^\circ})\resto {\sigma^{-1}(\Yi)} \simeq  (\etai)^*{\fii'}_*
    \big( \sLi' \big)^{[k]}  .  
  \end{equation}
  by \autoref{rem:ABC}.  After pulling back \autoref{eq:rest} to $C^\circ$ the
  left-hand side is $\sF\resto {C^\circ}$ and the right-hand-side is isomorphic to
  ${\sG_C}\resto {C^\circ}$ (\autoref{rem:ABC}).

  For the second claimed isomorphism, we note that by construction the two polarized
  families $(\wtilde U, \wtilde \sL)_{C'}$ and $((\Xii')_C, (\sLi')_C)_{C'}$ are
  isomorphic over the generic point of $C'$. With $\sM$ being separated we find the
  isomorphism
  \begin{equation}\label{isomSTAR}
    (\wtilde U, \wtilde \sL)_{C'}  \simeq  \big(  (\Xii')_{C}  , (\sLi')_C  \big)_{C'}  .
  \end{equation}
  Therefore, we have
  \begin{align*}
    \gamma^*\sG_C  &  =   \gamma^* (f_C)_* \big(  \sLi' \big)^{[k]}   \\
                   & \simeq (\eta')^* \big( \wtilde g_* \wtilde \sL^{[k]} \big)
                     \qquad \text{by \autoref{isomSTAR} and \autoref{rem:ABC}} \\
                   & = (\eta')^* \wtilde \sF . \qedhere
  \end{align*} 
\end{proof}
Now, according to \autoref{thm:Gab}, \autoref{claim:isoms} implies that there exists
a projective compactification $\widehat V$ of $V^\circ$, equipped with a locally free
sheaf $\widehat \sF$, a nonsingular projective scheme $V''$, and birational morphisms
$\rho: V''\to \widehat V$ and $\psi: V''\to \wtilde V$ such that $\widehat \sF$
restricted to $V^\circ$ is isomorphic to $\sF$ and
$\rho^*\widehat \sF\simeq \psi^*\wtilde \sF$.

On the other hand by base change (\autoref{rem:ABC}) and by the semipositivity
assumption on $\sM$, the restriction of $\wtilde \sF= \wtilde g_*(\wtilde \sL)^{[k]}$
to every irreducible curve in $\wtilde V$ is nef.  Since $\wtilde V$ is proper and
$\wtilde \sF$ is locally free (\autoref{rem:ABC}), this implies that $\wtilde \sF$ is
nef and therefore so is $\widehat \sF$.  This in turn implies that $\sF$ is weakly
positive over $V^\circ$. Furthermore, over $V^\circ$ we have
$\sF \simeq \sigma^*f_*\sL^{[k]}$. Thus by \autoref{item:splitflat} and
\autoref{item:split} we conclude that $f_*\sL^{[k]}$ is weakly positive over
$B\setminus D$.
This completes the proof of \autoref{thm:WPkey}.

\section{KSB-stability implies asymptotic stability}
\label{s:GIT}

In this section we apply \autoref{thm:WPkey} (established in \S \ref{ss:WPkey}) to prove 
\autoref{thm:main2}.
To do so we briefly recall the necessary background on Hilbert functors.

\subsection{Background on Hilbert functors}
\label{ss:Hilbert}

Let $\sM_{h} \subset \sM$ be a bounded moduli functor of polarized varieties with
Hilbert polynomial $h$.  Let $k_0\in \bN$ be as in \autoref{def:bounded}, so that for
any positive integer $k$ divisible by $k_0$ and for every $(Y,L)\in \Ob(\sM_h)$, the
natural map $\phi\resto{L^{[k]}|}: Y\to \bP^{r-1}$ is an embedding, where
$r:=r(k) = h^0(Y, L^{[k]})$, and that $h^i(Y, L^{[k]})=0$, for any $i>0$.  Assume
further that $\sM_h$ is locally closed (or constructible) and define the
\emph{Hilbert functor
  ${\sH}^{r,k}_{\sM_{h}} : \mathfrak{Sch}_{\bC} \to \mathfrak{Sets}$ of $k$-embedded
  schemes in $\sM_h$} by
\begin{equation*} {\sH}^{r,k}_{\sM_h}(B) = \Big\{ (f: X\to B, \sL)\in \sM_h(B)
  \bigm\mid \exist \; i:\bP_B(f_*\sL^{[k]}) \overset{\simeq\;}{\longrightarrow}
  \bP^{r-1}\times B \Big\} ,
\end{equation*}
so that, with $\sL^{[k]}$ being invertible (\autoref{rem:invertible}), the natural
map $\psi\resto{\sL^{[k]}/B|}: X\to \bP_B(f_*\sL^k)$ fiberwise defines an embedding
$(i\circ\psi\resto{\sL^{[k]}/B|})\resto {X_b}: X_b\to \bP^{r-1}$.

Next, let $H^{r,k}_h\subseteq \Hilb^{r-1}_h$ be the universal locally closed
(respectively constructible) subscheme corresponding to ${\sH}^{r,k}_{\sM_{h}}$.  Set
$f_{{\mathfrak X}}: {\mathfrak X} \subseteq \bP^{r-1} \times H^{r,k}_h \to H^{r,k}_h$
to denote the universal object of $\sH^{r,k}_{\sM_{h}}$, naturally defined by
restricting the universal family over $\Hilb_h^{r-1}$ to $H^{r,k}_h$, with
$\sL_{{\mathfrak{X}}}$ being its universal line bundle.  Set $\xi$ to denote the
natural map $\psi\resto{\sL^{[k]}_{{\mathfrak X}} / H^{r,k}_h |}$ composed with the
isomorphism
\[
  \bP_{H^{r,k}_h}\big( (f_{{\mathfrak X}})_* \sL^{[k]}_{{\mathfrak X}} \big) \simeq
  \bP^{r-1}\times H^{r,k}_h ;
\]
\[
  \xymatrix{ \mathfrak X \ar[dr]_{f_{\mathfrak X}} \ar@{^(->}[rr]^(.35){\xi} &&
    \bP^{r-1}\times H^{r,k} \ar[dl]^{\pr_2} \ar[r]^(.6){\pr_1} &\bP^{r-1}  \\
    & H^{r,k}_h }
\]
As such we have the relative isomorphism 
$\sL_{\mathfrak X}^{[k]}\simeq_{H^{r,k}} \xi^* \pr_1^*\big( \sO_{\bP^{r-1}}(1) \big)$, i.e., for a line bundle $\sB$ on $H^{r,k}_h$ we have 
\begin{equation}\label{eq:globaliso}
\sL_{\mathfrak X}^{[k]}  \simeq  \xi^* \pr_1^* \big( \sO_{\bP^{r-1}} (1) \big)  \otimes f_{\mathfrak X}^* (\sB) .
\end{equation}

Furthermore, let $\mu_0 \in \bN$ be any sufficiently large integer such that, 
for every $\mu\geq \mu_0$, 
\begin{equation}\label{eq:ample}
\det (f_{{\mathfrak X}})_* \xi^* \big( \pr_1^* \sO_{\bP^{r-1}} (\mu)  \big)
\end{equation}
is ample on $H^{r,k}$, 
\cite[Cor.~1.42]{Viehweg95}.

\begin{notation}\label{not:uni}
  To avoid cumbersome notations we will use $\sL$ in place of the universal line
  bundle $\sL_{\mathfrak X}$ on $\mathfrak X$, and $f$ to denote $f_{\mathfrak X}$.
\end{notation}

\noindent
Note that we have the naturally defined map:\vskip-1.75em
\[
  \big( \pr_2 \big)_* \pr_1^* \big( \sO_{\bP^{r-1}}(1) \big) \longrightarrow
  \overbrace{ \big( \pr_2 \big)_* (\xi)_*}^{(f_{\mathfrak X})_*} \xi^*
  \pr_1^*\big( \sO_{\bP^{r-1}} (1) \big) ,
\]
which together with \autoref{eq:globaliso}, and the isomorphism
$(\pr_2)_* \pr_1^* \sO_{\bP^{r-1}}(1) \simeq \bigoplus^{r} \sO_{H^{r,k}_h}$ leads to
the morphism
\begin{equation}\label{eq:doublestar}
  d\xi:  \bigoplus^{r} \sB  =  (\pr_2)_* \pr_1^* \big( \sO_{\bP^{r-1}}(1) \big)
  \otimes \sB \longrightarrow   f_* \sL^{[k]}  . 
\end{equation}

\begin{remark}[\protect{\cite[pp.~ 44--46]{Viehweg95}}]
\label{rem:iso}
The isomorphism
$\bP_{H^{r,k}_h}\big( f_* \sL^{[k]} \big) \simeq \bP^{r-1} \times H^{r,k}_h$ is
equivalent to $d\xi$ being an isomorphism.  In particular we have:
\begin{equation}\label{isom:B}
\det f _*  \sL^{[k]} \simeq  \sB^{r} .
\end{equation}
\end{remark}

\begin{proposition}\label{prop:isom}
For every integer $\mu\geq \mu_0$ the line bundle 
\begin{equation}\label{eq:standard}
\det  f_*  \sL^{[k\mu]}  \otimes  \Big(\det f_*  \sL^{[k]} \Big)^{-\mu }
\end{equation}
is ample on $H^{r,k}_h$. 
\end{proposition}
\ \vskip-3.5em \ 
\begin{proof}
By \autoref{eq:ample}, the line bundle 
\begin{equation*}
  \det f_*  \xi^* \big( \pr_1^* \sO_{\bP^{r-1}} (\mu)  \big)
  \overset{\autoref{eq:globaliso}}{\simeq}      
  \det \Big(f_* \sL^{[k\mu]} \otimes    \sB^{- \mu} \Big) 
  \simeq   \det f_*   \sL^{[k\mu]}  \otimes \sB^{-r\mu}
\end{equation*}
is ample on $H^{r,k}_h$. The rest now follows from \autoref{rem:iso}.
\end{proof}

\begin{terminology}
Being the most naturally defined ample line bundles on the Hilbert scheme, 
the line bundles \autoref{eq:standard} are referred to as standard linearizations
in GIT. 
\end{terminology}

\subsection{Group actions on Hilbert schemes}
\label{ss:Group}
Given $H^{r,k}_h\subseteq \Hilb^{r-1}_h$ and $\bP^{r-1}$ as in \S \ref{ss:Hilbert}, we know that 
changing coordinates in $\bP^{r-1}$ amounts to an automorphism of $H^{r,k}_h$, defining 
a group action $\sigma: G\times H^{r,k}_h \to H^{r,k}_h$, $G:= \Sl(r, \bC)$.

\begin{definition}[\protect{cf.~\cite[1.17]{Mum77}, \cite[Def.~8]{KolSurvey21}}]
\label{def:asym}
Let $\sM_h$ and $k_0$ be as in \S\ref{s:GIT}. For every integer $k>0$ divisible by $k_0$, 
let $\sB_k$ be an ample line bundle on $H^{k,r}$. 
We say $(Y,L)\in \Ob(\sM_h)$ is \emph{asymptotically (semi)stable} with respect to $\sB_k$, if for every $k$, the
associated closed point in $H^{r,k}_h$ is GIT-(semi)stable with respect to $\sB_k$. 
Furthermore, we say an object of $\sM_h$ is asymptotically (semi)stable (with respect to $\sB_k$) 
if it is represented by an asymptotically (semi)stable $(Y,L)$ as above. 
\end{definition}

\begin{assumption}\label{ass:proper}
In the rest of this subsection we will assume that the $G$-action on $H:= H^{r,k}_h$ is proper with finite stabilizers, i.e., the morphism 
$(\sigma, \pr_2): G\times H \to H\times H$ is proper (see for example \cite[Def.~3.1]{Viehweg95}). 
\end{assumption}

According to \cite[Lem.~4.15]{Viehweg95}, given a compactification $\overline H$ of $H$,
there is a quasi-projective scheme scheme $Z$ containing $G \times H$ as an open subscheme and proper 
morphisms 
\[
  \xymatrix
  { Z \ar[r]^-{\phi} \ar[d]^(.4){\rm p_2} & \overline H && \text{extending}
    && G    \times H \ar[r]^-{\sigma} \ar[d]^(.45){\pr_2} &  H    \\
    \overline H &&&&& H }
\]
such that 
\begin{enumerate}
\item $Z = U\cup V$, where $U := \phi^{-1}(H)$, $V:={\rm p}_2^{-1}(H)$, and 
\item $U\cap V= G\times H$.
\end{enumerate}

We will work in the setting of \S\ref{ss:Hilbert}. Denoting the pullback of $\mathfrak X$ through $\sigma$ and $\pr_2$ 
by respectively $\mathfrak X'$ and $\mathfrak X''$, the group action defined by $\sigma$ naturally corresponds to 
an isomorphism $\xi_{\mathfrak X}: \mathfrak X'\to \mathfrak X''$ over $G\times H$ \cite[p.~199]{Viehweg95}:
\[
\xymatrix{
\mathfrak X \ar[d]^{f}  &  \ar[l]_q \mathfrak X'' \ar[dr]^{f''}&& \ar[ll]_{\xi_{\mathfrak X}}  \mathfrak X' \ar[dl]_{f'}  \ar[r]^{\sigma'}  & \mathfrak X \ar[d]^{f}  \\
H                         && \ar[ll]_{\pr_2} G\times H  \ar[rr]^{\sigma}  &&  H 
}
\]

Setting $\sL':= (\sigma')^* \sL$, $\sL'':= q^*\sL$, and assuming that $\eta>0$ is
divisible by $k_0$ as defined in in \S\ref{ss:Hilbert}, we have
\[
  \sigma^*f_*\sL^{[\eta]} \simeq f'_*\big( (\sL')^{[\eta]}\big) \;\; \text{and} \;\;
  \pr_2^* f_*\sL^{[\eta]} \simeq f''_*\big((\sL'')^{[\eta]} \big).
\]
\begin{definition}\label{linBC}
  In the setting of \S \ref{ss:Hilbert} and \autoref{not:uni} we say that
  $f_*\sL^{[\eta]}$ is \emph{$G$-linearizable by base change}, if $\xi_{\mathfrak X}$
  is an isomorphism of polarized families, i.e. $\sL'\simeq \sL''$, inducing the
  isomorphism $f'_*\big((\sL')^{[\eta]}\big) \simeq f''_*\big((\sL'')^{[\eta]}\big)$
  and the resulting
  $\Phi: \sigma^* f_*\sL^{[\eta]} \overset{\simeq}{\longrightarrow} \pr_2^*
  f_*\sL^{[\eta]}$.  If there is a sufficiently large $\eta_0\in \bN$ for which this
  property holds for every $\eta>0$ divisible by $\eta_0$, then we refer to $\sM_h$
  as \emph{universally linearizable by base change}.
\end{definition}
\begin{remark}
We can similarly define the notion of linearizability by base change for $\big(f_*\sL^{[\eta]}\big)^{\otimes a} \otimes \big(\det f_*\sL^{[\eta']}\big)^b$, for any $a,b \geq 0$
and $\eta'\in \bN$ divisible by $k_0$.
Clearly if $\sM_h$ is linearizable by base change, then so is the latter sheaf. 
\end{remark}
\begin{proposition}\label{prop:LBC}
Every bounded, constructible and functorially polarized $\sM$
is universally linearizable by base change. 
\end{proposition}
\begin{proof}
The isomorphism $\sL'\simeq \sL''$ in \autoref{linBC} follows from the functorial choice of polarization. 
Moreover by  \autoref{rem:ABC} 
we know, for $k_0>0$ as in \autoref{def:bounded} and every $k$ divisible by $k_0$, $f_*\sL^{[k]}$ 
commutes with base change. 
\end{proof}

\begin{example}\label{ex:KSB}
  $\sM^{\KSBm}_h$ is universally linearizable by base change by \autoref{KSB},
  \autoref{rem:KSB}, and \autoref{prop:LBC}.
\end{example}

\subsection{A stability criterion for linearizations over Hilbert schemes}
\begin{notation}\label{not:A}
  For any polarized family $f:(X, \sL)\to B$, and any $m\in \bZ$, we define
  $\sA_m(f, \sL):= \det f_*\sL^{[m]}$.
\end{notation}

\noindent
Working in the setting of \S \ref{ss:Hilbert}, \autoref{not:A}, \autoref{not:uni},
and \S \ref{ss:Group}, we consider the universal family $f:(\mathfrak X, \sL)\to H$
of $\sH^{r,k}_{\sM_h}$.

\begin{theorem}\label{thm:ViehwegStability}
  Assume that $\sM_h$ is bounded, separated, locally closed or constructible, and
  that $|\Aut(Y,L_Y)|<\infty$, for every $(Y,L_Y)\in \Ob(\sM_h)$. Further assume that
  the following two conditions hold:
  \begin{enumerate}
  \item \label{item:linBC} The functor $\sM_h$ is functorially polarized.
  \item \label{item:WP} There is a $k_0>0$, such that $h_*\sN^{[k]}$ is weakly
    positive over $V$ for every $k>0$ divisible by $k_0$ and each
    $\big( h:(U, \sN) \to V \big) \in \sM_h(V)$, for any reduced quasi-projective
    scheme $V$.
  \end{enumerate}
  Then, for every $\mu \geq \mu_0$, the line bundle
  \begin{equation}\label{eq:lin}
    \sA_{k}(f, \sL) \otimes \sA^r_{k\mu} (f, \sL) 
  \end{equation}
  defines a $G$-linearized ample line bundle with respect to which the entire $H$ is
  GIT-stable.
\end{theorem}
\begin{proof}
  We start by the observation that \autoref{ass:proper} is satisfied by the
  finiteness assumption on the polarized automorphism groups and separatedness of
  $\sM_h$ \cite[Lem.~7.6]{Viehweg95}.  By \cite[Prop.~3.36]{Viehweg95} we may also
  assume that $H$ is reduced.

  By \autoref{eq:doublestar} and \autoref{rem:iso} we know that
  $f_*\sL^{[k]}\simeq \bigoplus^r \sB$.  Therefore $\sB$ is weakly positive over $H$.
  On the other hand, by \autoref{prop:isom} we know that
  \[
    \sA_{k\mu}(f,\sL) \otimes \sA_k^{-\mu} (f,\sL)
  \]
  is ample over $H$ and thus, using the assumption that $\sA_k(f, \sL)$ is weakly
  positive over $H$, so is $\sA_{k\mu}(f,\sL)$.  Consequently
  $\sB\otimes \sA_{k\mu}(f, \sL)$ is ample.

  Next, consider the pullback $(\mathfrak X_U,\sL_U):= (\mathfrak X, \sL)_U$ and
  $(\mathfrak X_V, \sL_V) : = (\mathfrak X, \sL)_V$ via $\phi\resto U: U\to H$ and
  ${p_2}\resto V: V\to H$, respectively.  Using the isomorphism $\xi_{\mathfrak X}$
  over $U\cap V$ in \S \ref{ss:Group} and the assumption on the functoriality of the
  polarization $\sL$ we have the isomorphism
  \[
    \big( \mathfrak X_U, \sL_U\big)_{U\cap V} \simeq \big(\mathfrak X_{V} , \sL_V
    \big)_{U\cap V}.
  \]
  Therefore, using the fact that $Z= U\cup V$, after gluing together, we find the
  polarized family $\big( g:(Y, \sL_Y)\to Z \big) \in \sM(Z)$.

  \begin{claim}\label{claim:linBC}
    The two locally free sheaves
    \[
      \qquad f_*\sL^{[k]} \quad \text{and} \quad f_*\sL^{[k]}\otimes
      \sA_{k\mu}(f,\sL) \simeq \bigoplus^r \big( \sB\otimes \sA_{k\mu}(f,\sL) \big)
    \]
    are linearizable by base change. In particular $\sA_{k\mu}(f, \sL)$ and $\sB$ are
    $G$-linearized.
  \end{claim}
  \begin{proof}[Proof of \autoref{claim:linBC}]
    This directly follows from the assumption on the chosen linearization and
    \autoref{prop:LBC}.
  \end{proof}

  \begin{claim}\label{claim:final}
    With $\sF:= g_*\sL^{[k]}_Y \otimes \sA_{k\mu}(g, \sL_Y)$ there are two
    isomorphisms
    \[
      \xymatrix{ \gamma_U : \big(\phi\resto U\big)^* \Big( f_*\sL^{[k]} \otimes
        \sA_{k\mu}(f, \sL) \Big) \overset{\simeq}{\longrightarrow} \sF\resto U }
    \]
    \[
      \xymatrix{ \gamma_V : \big({p_2}\resto V\big)^* \Big( f_*\sL^{[k]} \otimes
        \sA_{k\mu}(f, \sL) \Big) \overset{\simeq}{\longrightarrow} \sF\resto V }
    \]
    such that after restriction to $U\cap V$ the linearization of
    $f_*\sL^{[k]} \otimes \sA_{k\mu}(f, \sL)$ induced by
    \[
      \big(\gamma_V\resto {U\cap V} \big)^{-1} \circ \gamma_U\resto {U\cap V}
    \]
    coincides with the one defined by base change in \autoref{claim:linBC}.
  \end{claim}
  \begin{proof}[Proof of \autoref{claim:final}]
    Using the isomorphism $(Y, \sL_Y)_U \simeq (\mathfrak X , \sL)_U$ (defined by
    pulling back through $\phi\resto U: U\to H$) and base change we can define
    $\gamma_U$ by:
    \[
      (\phi\resto U)^* \Big( f_*\sL^{[k]} \otimes \sA_{k\mu}(f, \sL) \Big)
      \overset{\gamma_U}{\simeq} \big( g\resto U \big)_* \sL^{[k]}_Y \otimes
      \sA_{k\mu}(g, \sL_Y)\resto U = \sF\resto yU
    \]
    The isomorphism $\gamma_V$ can be similarly defined. The rest follows from the
    construction.
  \end{proof}
  Now, by applying \cite[Thm.~4.25]{Viehweg95} to $\sB\otimes \sA_{k\mu}(f,\sL)$ and
  using the fact that $\sF$ is weakly positive over $Z_{\red}$ (Item~\ref{item:WP}),
  we find that the linearized line bundle $\sB\otimes \sA_{k\mu}$
  (\autoref{claim:linBC}) is stable over the entire $H$ and thus so is
  \[
    \big( \sB\otimes \sA_{k\mu}(f,\sL)\big)^r \overset{\autoref{isom:B}}{\simeq}
    \sA_k(f,\sL) \otimes \sA^r_{k\mu} (f, \sL) .\qedhere
  \]
\end{proof}

\subsection{Proof of \autoref{thm:main2}}
This follows directly from \autoref{thm:WPkey} and 
\autoref{thm:ViehwegStability}.

\subsection{Proof of \autoref{thm:main1}}\label{proofmain1}
With $I=\{ h_i \in \bQ[x] \}$ as in \autoref{rem:KSB}, it suffices to prove the
theorem for each $\sM^{\KSBm}_{h_i}$, which is bounded by \cite{HMX18} and separated
by uniqueness of canonical models.  Now, according to \cite[Thm.~1.2]{Fuj18}
${\sM}_{h_i}^{\KSBm}$ is semipositive.  On the other hand, by \cite{HX13}
${\mathcal M}^{\KSBm}_{h_i}$ is proper. Therefore, by \autoref{thm:WPkey} (more
specifically \autoref{cor:WP}) Item \autoref{item:WP} is satisfied.  With the
canonical choice of polarization, Item~\autoref{item:linBC} also holds
(\autoref{ex:KSB}).  The rest now follows from \autoref{thm:ViehwegStability}.


\begin{bibdiv}
\begin{biblist}





\bibitem[AV02]{AV02}
{\sc D.~Abramovich and A.~Vistoli}: \emph{Compactifying the space of stable
  maps}, J. Amer. Math. Soc. \textbf{15} (2002), 27--75,
  \href{https://doi.org/10.1090/S0894-0347-01-00380-0}{DOI:10.1090/S0894-0347-01-00380-0}.

\bibitem[Alp13]{MR3237451}
{\sc J.~Alper}: \emph{Good moduli spaces for {A}rtin stacks}, Ann. Inst.
  Fourier (Grenoble) \textbf{63} (2013), 6, 2349--2402.

\bibitem[BCHM10]{BCHM10}
{\sc C.~Birkar, P.~Cascini, C.~D. Hacon, and J.~McKernan}: \emph{Existence of
  minimal models for varieties of log general type}, Journal of the AMS
  \textbf{23} (2010), 405--468,

\bibitem[Cam04]{Cam04}
{\sc F.~Campana}: \emph{Orbifolds, special varieties and classification
  theory}, Ann. Inst. Fourier (Grenoble) \textbf{54} (2004), 3, 499--630.

\bibitem[Con00]{Conrad00}
{\sc B.~Conrad}: \emph{Grothendieck duality and base change}, Lecture Notes in
  Mathematics, vol. 1750, Springer-Verlag, Berlin, 2000.

\bibitem[Don01]{Don01}
{\sc S.~K. Donaldson}: \emph{Scalar curvature and projective embeddings}, J.
  Differential Geom. \textbf{59} (2001), 3, 479--522.

\bibitem[Fuj18]{Fuj18} {\sc O.~Fujino}: \emph{Semipositivity theorems for moduli
    problems}, Ann. Math.  \textbf{187} (2018), 3, 639--665,

\bibitem[Gie82]{Gie82}
  {\sc D.~Gieseker}: \emph{Lecture on moduli of curves}, Tata Institute of
  Fundamental Research Lectures on Mathematics and Physics, vol.~69,
  1982.

\bibitem[HMX12]{HMX13}
{\sc C.~D. Hacon, J.~McKernan, and C.~Xu}: \emph{{ACC} for log canonical
  thresholds}, Preprint \href{http://arxiv.org/abs/1208.4150}{arXiv:1208.4150}.
  To appear in Annals of Mathematics., August 2012.

\bibitem[HMX18]{HMX18}
{\sc C.~D. Hacon, J.~McKernan, and C.~Xu}: \emph{Boundedness of moduli of
  varieties of general type}, J. Eur. Math. Soc. (JEMS) \textbf{20} (2018), 4,
  523--571, \href{https://arxiv.org/abs/1412.1186}{arxiv.org/abs/1412.1186}.

\bibitem[HX13]{HX13}
{\sc C.~D. Hacon and C.~Xu}: \emph{Existence of log canonical closures},
  Invent. Math. \textbf{192} (2013), 1, 161--195.

\bibitem[Har66]{Ha66}
  {\sc R.~Hartshorne}: \emph{Ample vector bundles}, Publ. Math.
  IHES \textbf{29} (1966), 63--94.

\bibitem[Har77]{Ha77}
{\sc R.~Hartshorne}: \emph{Algebraic geometry}, Springer-Verlag, New York,
  1977, Graduate Texts in Mathematics, No. 52.

\bibitem[Kaw81]{Kaw81} {\sc Y.~Kawamata}: \emph{{Characterization of Abelian
      Varieties}}, Comp.~Math.~ \textbf{43} (1981), 253--276.

\bibitem[KM97]{KM97}
{\sc S.~Keel and S.~Mori}: \emph{Quotients by groupoids}, Annals Math.
  \textbf{145} (1997), 1, 193--213,

\bibitem[Kol87]{Kollar87} {\sc J.~Koll{\'a}r}: \emph{Subadditivity of the {K}odaira
    dimension: f{i}bers of general type}, Algebraic geometry, {S}endai, 1985,
  Adv. Stud. Pure Math., vol.~10, 1987, pp.~361--398.

\bibitem[Kol90]{Kollar90}
{\sc J.~Koll{\'a}r}: \emph{Projectivity of complete moduli}, J. Differential
  Geom. \textbf{32} (1990), 1, 235--268.

\bibitem[Kol97]{Kollar97a} {\sc J.~Koll{\'a}r}: \emph{Quotient spaces modulo
    algebraic groups}, Ann.~of Math.(2) \textbf{145} (1997), 1, 33--79.

\bibitem[Kol13]{KollarSingsOfTheMMP} {\sc J.~Koll{\'a}r}: \emph{Singularities of the
    minimal model program}, Cambridge Tracts in Mathematics, vol. 200, CUP, 2013,
  With a collaboration of S{\'a}ndor Kov{\'a}cs.

\bibitem[Kol21]{KolSurvey21}
{\sc J.~Koll\'ar}: \emph{Mumford’s influence on the moduli theory of
 algebraic varieties}, Pure Appl. Math. Q. \textbf{17} (2021), 2, 619--647,
  \href{https://doi.org/10.4310/PAMQ.2021.v17.n2.a3}{DOI:10.4310/PAMQ.2021.v17.n2.a3}.

\bibitem[Kol23]{ModBook} {\sc J.~Koll\'{a}r}: \emph{Families of varieties of general
    type}, Cambridge Tracts in Mathematics, vol. 231, CUP, 2023, With the
  collaboration of Klaus Altmann and S{\'a}ndor Kov{\'a}cs.

\bibitem[KK10]{KK10}
{\sc J.~Koll{\'a}r and S.~J Kov{\'a}cs}: \emph{Log canonical singularities are
  {D}u {B}ois}, J. Amer. Math. Soc. \textbf{23} (2010), no.~3, 791--813.
  
\bibitem[KSB88]{Ksb88}
{\sc J.~Koll{\'a}r and N.~I. Shepherd-Barron}: \emph{Threefolds and
  deformations of surface singularities}, Invent. Math. \textbf{91} (1988), 2,
  299--338.

\bibitem[Kov13]{Kov13}
{\sc S.~J Kov{\'a}cs}: \emph{Singularities of stable varieties}, Handbook of
  moduli. Adv. Lect. Math. (ALM), 25, vol.~II, Int. Press, Somerville, MA,
  2013.

\bibitem[KP17]{KP17}
{\sc S.~J Kov\'acs and {\relax Zs}.~Patakfalvi}: \emph{Projectivity of the
  moduli space of stable log-varieties and subadditivity of log-{K}odaira
  dimension}, J. Amer. Math. Soc. \textbf{30} (2017), 4, 959--1021.

\bibitem[LMB00]{LMB00} {\sc G.~Laumon and L.~Moret-Bailly}: \emph{Champs
    alg{\'e}briques}, Ergeb.~Math. vol.~39, Springer-Verlag, 2000.
  

\bibitem[Laz14]{Laz14} {\sc R.~Laza}: \emph{Perspectives on the construction and
    compactification of moduli spaces.}, Paul Hacking et. al, Compactifying moduli
  spaces, 1--39, Adv. Courses in Math.
  2016
 

\bibitem[Laz04]{Laz04-I} {\sc R.~Lazarsfeld}: \emph{Positivity in algebraic
    geometry. {I}}, Ergebnisse der Mathematik und ihrer Grenzgebiete. 3. Folge,
  vol.~48, Springer-Verlag, 2004. 

\bibitem[Luo98]{Luo98}
{\sc H.~Luo}: \emph{Geometric criterion for {G}ieseker-{M}umford stability of
  polarized manifolds}, J. Differential Geom. \textbf{49} (1998), 3, 577--599.

\bibitem[Mum77]{Mum77} {\sc D.~Mumford}: \emph{Stability of projective varieties},
  Enseig.~Math.~\textbf{1-2} (1977), 23, 39--110.

\bibitem[MFK94]{MR1304906} {\sc D.~Mumford, J.~Fogarty, and F.~Kirwan}:
  \emph{Geometric invariant theory}, third ed., Ergebnisse der Mathematik und ihrer
  Grenzgebiete (2), vol.~34, Springer-Verlag, Berlin, 1994.

\bibitem[Ols16]{Olsson16}
{\sc M.~Olsson}: \emph{Algebraic spaces and stacks},
  Colloquium Publications, vol.~62, American Mathematical Society, 2016.

\bibitem[Ryd13]{Ryd13}
{\sc D.~Rydh}: \emph{Existence and properties of geometric quotients}, J.
  Algebraic Geom. \textbf{22} (2013), 629--669,
  \href{https://doi.org/10.1090/S1056-3911-2013-00615-3}{DOI:S1056-3911-2013-00615-3}.

\bibitem[StacksProject]{stacks-project}
{\sc {Stacks Project Authors}}: \emph{Stacks {P}roject},
  \url{http://stacks.math.columbia.edu}.

\bibitem[Vie95]{Viehweg95}
{\sc E.~Viehweg}: \emph{Quasi-projective moduli for polarized manifolds},
  Ergebnisse der Mathematik und ihrer Grenzgebiete (3), vol.~30,
  Springer-Verlag, Berlin, 1995.

\bibitem[Vie83a]{Viehweg83}
{\sc E.~Viehweg}: \emph{Weak positivity and the additivity of the {K}odaira
  dimension for certain fibre spaces}, Algebraic varieties and analytic
  varieties (Tokyo, 1981), Adv. Stud. Pure Math., vol.~1, North-Holland,
  Amsterdam, 1983, pp.~329--353.

\bibitem[Vie83b]{Vie83}
{\sc E.~Viehweg}: \emph{Weak positivity and the additivity of the {K}odaira
  dimension. {II}. {T}he local {T}orelli map}, Classification of algebraic and
  analytic manifolds (Katata, 1982), Progr. Math., vol.~39, Birkh\"auser
  Boston, Boston, MA, 1983, pp.~567--589.

\bibitem[Vie89]{MR996558}
{\sc E.~Viehweg}: \emph{Weak positivity and the stability of certain {H}ilbert
  points}, Invent. Math. \textbf{96} (1989), 3, 639--667.

\bibitem[Vie10]{Viehweg06}
{\sc E.~Viehweg}: \emph{Compactifications of smooth families and of moduli
  spaces of polarized manifolds},
  \href{https://doi.org/10.4007/annals.2010.172.809}{DOI:10.4007/annals.2010.172.809},
  2010.

\bibitem[Wan04]{Wan04}
{\sc X.~Wang}: \emph{Futaki invariant and stability of projective manifolds},
  Comm. Anal. Geom. \textbf{12} (2004), 5, 1009--1037.

\bibitem[WX14]{WX14}
{\sc X.~Wang and C.~Xu}: \emph{Nonexistence of asymptotic {GIT}
  compactification}, Duke. Math. \textbf{163} (2014), 12, 2217--2241.

\bibitem[Yau77]{MR0451180}
{\sc S.-T. Yau}: \emph{Calabi's conjecture and some new results in algebraic
  geometry}, Proc. Nat. Acad. Sci. U.S.A. \textbf{74} (1977), 5, 1798--1799.
\end{biblist}
\end{bibdiv}

\end{document}
